\documentclass[11pt]{elsarticle}
\usepackage{amsmath, amsfonts, amsthm, enumerate, pmat, tikz, tkz-berge, booktabs}

\newtheorem{thm}{Theorem}[section]
\newtheorem{cor}[thm]{Corollary}
\newtheorem{lem}[thm]{Lemma}

\newtheorem{obv}[thm]{Observation}

\theoremstyle{definition}
\newtheorem{defin}[thm]{Definition}
\theoremstyle{remark}
\newtheorem*{rem}{Remark} %unnumbered remarks
\theoremstyle{definition}
\newtheorem{ex}[thm]{Example}

 \definecolor{helena}{rgb}{.2,.8,.4}

\makeatletter
\def\ps@pprintTitle{%
 \let\@oddhead\@empty
 \let\@evenhead\@empty
 \def\@oddfoot{}%
 \let\@evenfoot\@oddfoot}
\makeatother
\begin{document}

\allowdisplaybreaks

\begin{frontmatter}

\title{Connecting sufficient conditions for the Symmetric Nonnegative Inverse Eigenvalue Problem}

\author{Richard Ellard\corref{FundingThanks}}
\ead{richardellard@gmail.com}

\author{Helena \v{S}migoc\corref{FundingThanks}}
\ead{helena.smigoc@ucd.ie}

\cortext[FundingThanks]{The authors' work was supported by Science Foundation Ireland under Grant 11/RFP.1/MTH/3157.}

\address{
	School of Mathematical Sciences, \\ 
	University College Dublin, \\ 
	Belfield, Dublin 4, Ireland
}

\begin{abstract}
We say that a list of real numbers is ``symmetrically realisable'' if it is the spectrum of some (entrywise) nonnegative symmetric matrix. The Symmetric Nonnegative Inverse Eigenvalue Problem (SNIEP) is the problem of characterising all symmetrically realisable lists. 
%If we do not insist that the realising matrix be symmetric, then the resulting problem is called the Real Nonnegative Inverse Eigenvalue Problem (RNIEP).

In this paper, we present a recursive method for constructing symmetrically realisable lists. The properties of the realisable family we obtain allow us to make several novel connections between 
a number of sufficient conditions developed over forty years, starting with the work of Fiedler in 1974. We show that essentially all previously known sufficient conditions are either contained in or equivalent to the family we are introducing.  
 
%In this work, we unify several approaches to the SNIEP/RNIEP. Specifically, we survey two methods of constructing symmetrically realisable lists: one due to Soules [1983] and later generalised by Elsner, Nabben and Neumann [1998], and one due to Soto [2013]. We also consider a condition due to Borobia, Moro and Soto [2008] called ``C-realisability'' which is sufficient for the existence of a not-necessarily-symmetric realising matrix. We give a recursive method for constructing symmetrically realisable lists based on a construction of \v{S}migoc [2004] and we use this method to show that the symmetrically realisable lists obtainable by Soules and Soto are identical and that these lists are precisely the C-realisable lists. As a corollary, we see that C-realisability is also sufficient for the SNIEP. By viewing these symmetrically realisable lists through the lens of our recursive method, several interesting properties are also revealed.
\end{abstract}

\begin{keyword}
	Nonnegative matrices \sep Symmetric Nonnegative Inverse Eigenvalue Problem \sep Soules matrix
	\MSC[2010] 15A18 \sep 15A29
\end{keyword}

\end{frontmatter}

\section{Introduction}

This paper explores the spectral properties of symmetric nonnegative matrices. 
Nonnegative matrices were a topic of special interest of Hans Schneider: he had over fifty papers in the area, the most relevant of these to our present paper being \cite{MR3217406,MR1780191}.

Let $\sigma:=(\lambda_1,\lambda_2,\ldots,\lambda_n)$ be a list of $n$ real numbers. If there exists a nonnegative symmetric matrix $A$ with spectrum $\sigma$, then we say $\sigma$ is \emph{symmetrically realisable} and that $A$ realises $\sigma$. The \emph{Symmetric Nonnegative Inverse Eigenvalue Problem} (SNIEP) is the problem of characterising all symmetrically realisable lists.

Since the spectrum of a symmetric matrix is necessarily real, the restriction that $\sigma$ consist only of real numbers is a natural one; however, if we allow $A$ to be not-necessarily-symmetric, but consider only lists of real numbers, then the resulting problem is known as the \emph{Real Nonnegative Inverse Eigenvalue Problem} (RNIEP).

In this paper, we describe a recursive method of constructing symmetrically realisable lists, using a construction of \v{S}migoc \cite{SmigocDiagonalElement}. The properties of the realisable lists obtainable in this way allow us to show that essentially all known sufficient conditions to date are either contained in or equivalent to the realisability we are introducing. This includes the method of Soules \cite{Soules} (later generalised by Elsner, Nabben and Neumann \cite{ElsnerNabbenNeumann}), one of the most influential methods of constructing symmetrically realisable lists. Moreover, since we also show that the realising matrices we obtain by our method have the same form as the ones obtained by the method of Soules, our approach gives a new insight into Soules realisability.

%This includes one of the most influential methods for constructing symmetrically realisable lists due to Soules \cite{Soules} (later generalised by Elsner, Nabben and Neumann \cite{ElsnerNabbenNeumann}). 
%
%
%We describe an alternative recursive method of constructing symmetrically realisable lists, using a construction of \v{S}migoc \cite{SmigocDiagonalElement}. We show that this recursive method is equivalent to the method of Soules in that the symmetrically realisable lists obtained are identical and the realising matrices have the same form.

 We also consider a sufficient condition for the RNIEP due to Borobia, Moro and Soto \cite{UnifiedView} called ``C-realisability'' and a family of sufficient conditions for the SNIEP due to Soto \cite{Soto2013}. We show that C-realisability is also sufficient for the SNIEP and that $\sigma$ is C-realisable if and only if it satisfies one of Soto's conditions. Such $\sigma$ are precisely those which may be obtained by our method or the method of Soules. The equivalence of all four methods is proved in Section \ref{sec:Equivalence}.

In Section \ref{sec:SNIEPPreliminaries}, we outline the background to and terminology used in this paper. In Section \ref{sec:Hn}, we describe our recursive approach and prove several properties of the realisable lists which may be obtained in this manner. Section \ref{sec:LiteratureComparison} can be seen as a survey of sufficient conditions for the SNIEP given in the literature, including Suleimanova \cite{Suleimanova1949}, Perfect \cite{Perfect1953}, Ciarlet \cite{Ciarlet1968}, Kellog \cite{Kellog1971}, Salzmann \cite{Salzmann1972}, Fiedler \cite{Fiedler}, Borobia \cite{Borobia1995} and Soto \cite{Soto2003}. We show that if $\sigma$ obeys any of these sufficient conditions, then $\sigma$ may also be obtained by our method.

\section{Preliminaries and notation}\label{sec:SNIEPPreliminaries}

To denote that $\sigma$ is symmetrically realisable, we may sometimes write $\sigma\in\mathcal{R}_n$ . In this paper, the diagonal elements of the realising matrix will also be important; hence, if there exists a nonnegative symmetric matrix $A$ with diagonal elements $(a_1,a_2,\ldots,a_n)$ and specrum $\sigma$, then we write
\[
	\sigma\in\mathcal{R}_n(a_1,a_2,\ldots,a_n).
\]
If we wish to specify that $\lambda_1$ is the Perron eigenvalue of the realising matrix, we will separate $\lambda_1$ from the remaining entries in the list by a semicolon, e.g. we may write
\[
	(\lambda_1;\lambda_2,\ldots,\lambda_n)\in\mathcal{R}_n
\]
or
\[
	(\lambda_1;\lambda_2,\ldots,\lambda_n)\in\mathcal{R}_n(a_1,a_2,\ldots,a_n).
\]
The remaining eigenvalues $\lambda_2,\lambda_3,\ldots,\lambda_n$ will generally be considered unordered. The diagonal elements $a_1,a_2,\ldots,a_n$ will also generally be cosidered unordered and they may appear in any order on the diagonal of $A$, i.e. we do not assume that $a_i$ is the $(i,i)$ entry of $A$. Sometimes we will assume that the $\lambda_i$ or $a_i$ are arranged in non-increasing order and if this is the case, we will say so explicitly. In this paper, $\mathcal{R}$ will always be replaced by either $\mathcal{S}$ or $\mathcal{H}$, depending on whether we are considering realisability via Soules or our recursive method.

We begin by stating some necessary conditions (due to Fiedler \cite{Fiedler}) for $\sigma$ to be the spectrum of a nonnegative symmetric matrix with specified diagonal elements:

\begin{thm}{\bf\cite{Fiedler}}\label{thm:FiedlerNecessary}
	If $\lambda_1\geq\lambda_2\geq\cdots\geq\lambda_n$, $a_1\geq a_2\geq\cdots\geq a_n\geq0$ and $(\lambda_1,\lambda_2,\ldots,\lambda_n)$ is the spectrum of a nonnegative symmetric matrix with diagonal elements $(a_1,a_2,\ldots,a_n)$, then
	\begin{gather*}
		\lambda_1\geq a_1,\\
		\sum_{i=1}^n\lambda_i=\sum_{i=1}^na_i
	\end{gather*}
	and
	\[
		\sum_{i=1}^s\lambda_i+\lambda_k\geq\sum_{i=1}^{s-1}a_i+a_{k-1}+a_k
	\]
	for all $1\leq s<k\leq n$ (with the convention that $\sum_{i=1}^0a_i=0$).
\end{thm}

Fiedler also gave the following sufficient conditions:

\begin{thm}{\bf\cite{Fiedler}}\label{thm:FiedlerSufficient}
	Let $\lambda_1\geq\lambda_2\geq\cdots\geq\lambda_n$ and $a_1\geq a_2\geq\cdots\geq a_n\geq0$ satisfy the following conditions:
	\begin{gather}
		\sum_{i=1}^k\lambda_i\geq\sum_{i=1}^ka_i :\:\: k=1,2,\ldots,n-1,\notag \\
		\sum_{i=1}^n\lambda_i=\sum_{i=1}^na_i,\notag \\
		\lambda_k\leq a_{k-1} :\:\: k=2,3,\ldots,n-1.\label{eq:SoulesSufficient}
	\end{gather}
	Then  $(\lambda_1,\lambda_2,\ldots,\lambda_n)$ is the spectrum of a nonnegative symmetric matrix with diagonal elements $(a_1,a_2,\ldots,a_n)$.
\end{thm}

For $n\leq3$, the question of whether $\sigma\in\mathcal{R}_n(a_1,a_2,\ldots,a_n)$ is completely solved by Theorems \ref{thm:FiedlerNecessary} and \ref{thm:FiedlerSufficient}. If $n=2$, the matrix
\[
	\left[
		\begin{array}{cc}
			a_1 & \sqrt{(\lambda_1-a_1)(\lambda_1-a_2)} \\
			\sqrt{(\lambda_1-a_1)(\lambda_1-a_2)} & a_2
		\end{array}
	\right]
\]
has spectrum $(\lambda_1,a_1+a_2-\lambda _1)$ and hence if $\lambda_1\geq\lambda_2$ and $a_1\geq a_2\geq0$, then $(\lambda_1,\lambda_2)$ is the spectrum of a nonnegative symmetric matrix with diagonal elements $(a_1,a_2)$ if and only if the following conditions are satisfied:
\begin{equation}\label{eq:n=2NS}
	\left\{\begin{array}{c}
		\lambda_1\geq a_1,\\
		\lambda_1+\lambda_2=a_1+a_2.
	\end{array}\right.
\end{equation}
If $n=3$, then the conditions of Theorems \ref{thm:FiedlerNecessary} and \ref{thm:FiedlerSufficient} are identical and hence if $\lambda_1\geq\lambda_2\geq\lambda_3$ and $a_1\geq a_2\geq a_3\geq0$, then $(\lambda_1,\lambda_2,\lambda_3)$ is the spectrum of a nonnegative symmetric matrix with diagonal elements $(a_1,a_2,a_3)$ if and only if the following conditions are satisfied:
\begin{equation}\label{eq:n=3NS}
	\left\{\begin{array}{c}
		\lambda_2\leq a_1\leq\lambda_1,\\
		\lambda_3\leq a_3,\\
		\lambda_1+\lambda_2+\lambda_3=a_1+a_2+a_3.
	\end{array}\right.
\end{equation}

\subsection{The Soules approach to the SNIEP}

Soules' approach to the SNIEP focuses on constructing the eigenvectors of the realising matrix $A$. Starting from a positive vector $x\in\mathbb{R}^n$, Soules \cite{Soules} showed how to construct a real orthogonal $n\times n$ matrix $R$ with first column $x$ such that for all $\lambda_1\geq\lambda_2\geq\cdots\geq\lambda_n\geq0$, the matrix $R\Lambda R^T$---where $\Lambda:=\mathrm{diag}(\lambda_1,\lambda_2,\ldots,\lambda_n)$---is nonnegative. This motivated Elsner, Nabben and Neumann \cite{ElsnerNabbenNeumann} to make the following definition:

\begin{defin}
	Let $R\in\mathbb{R}^{n\times n}$ be an orthogonal matrix with columns $r_1, r_2,\ldots,r_n$. $R$ is called a \emph{Soules matrix} if $r_1$ is positive and for every diagonal matrix $\Lambda:=\text{diag}(\lambda_1,\lambda_2,\ldots,\lambda_n)$ with $\lambda_1\geq\lambda_2\geq\cdots\geq\lambda_n\geq0$, the matrix $R\Lambda R^T$ is nonnegative.
\end{defin}

With regard to the SNIEP, a key property of Soules matrices is the following:

\begin{thm}{\bf \cite{ElsnerNabbenNeumann}}
	Let $R$ be a Soules matrix and let $\Lambda:=\mathrm{diag}(\lambda_1, \lambda_2, \ldots, ${ }$\lambda_n)$, where $\lambda_1\geq\lambda_2\geq\cdots\geq\lambda_n$. Then the off-diagonal entries of the matrix $R\Lambda R^T$ are nonnegative.
\end{thm}

Therefore, if $R=(r_{ij})$ is an $n\times n$ Soules matrix and $\Lambda:=\text{diag}(\lambda_1,${ }$\lambda_2,\ldots,\lambda_n)$, where $\lambda_1\geq\lambda_2\geq\cdots\geq\lambda_n$, then $\sigma:=(\lambda_1,\lambda_2,\ldots,\lambda_n)$ is the spectrum of a nonnegative symmetric matrix if the diagonal elements of $R\Lambda R^T$ are nonnegative. This motivates the following definition:

\begin{defin}
	Let $\lambda_1\geq\lambda_2\geq\cdots\geq\lambda_n$ and let $a_1,a_2,\ldots,a_n\geq0$. We write
	\begin{equation}\label{eq:SDef}
		(\lambda_1;\lambda_2,\ldots,\lambda_n)\in\mathcal{S}_n(a_1,a_2,\ldots,a_n)
	\end{equation}
	if there exists an $n\times n$ Soules matrix $R$ such that the matrix $R\Lambda R^T$---where $\Lambda:=\mathrm{diag}(\lambda_1,\lambda_2,\ldots,\lambda_n)$---has diagonal elements $(a_1,a_2,\ldots,${ }$a_n)$. We write
	\[
		(\lambda_1,\lambda_2,\ldots,\lambda_n)\in\mathcal{S}_n
	\]
	if there exist $a_1,a_2,\ldots,a_n\geq0$ such that (\ref{eq:SDef}) holds and we call $\mathcal{S}_n$ the \emph{Soules set}.
\end{defin}

Elsner, Nabben and Neumann generalised the work of Soules by characterising \emph{all} Soules matrices. In order to state their characterisation, we require two definitions:

\begin{defin}
	Let $\mathcal{N}=(\mathcal{N}_1,\mathcal{N}_2,\ldots,\mathcal{N}_n)$ be a sequence of partitions of $\{1,2,\ldots,n\}$. We say that $\mathcal{N}$ is \emph{Soules-type} if $\mathcal{N}$ has the following properties:
	\begin{enumerate}[(i)]
		\item for each $i\in\{1,2,\ldots,n\}$, the partition $\mathcal{N}_i$ consists of precisely $i$ subsets, say $\mathcal{N}_i=\{\mathcal{N}_{i,1},\mathcal{N}_{i,2},\ldots,\mathcal{N}_{i,i}\}$;
		\item for each $i\in\{2,3,\ldots,n\}$, there exist indices $j,k,l$ with $1\leq j\leq i-1$ and $1\leq k<l\leq i$, such that $\mathcal{N}_{i-1}\setminus\mathcal{N}_{i-1,j}=\mathcal{N}_i\setminus\{\mathcal{N}_{i,k},\mathcal{N}_{i,l}\}$ and $\mathcal{N}_{i-1,j}=\mathcal{N}_{i,k}\cup\mathcal{N}_{i,l}$, i.e. $\mathcal{N}_i$ is constructed from $\mathcal{N}_{i-1}$ by splitting one of the sets $\mathcal{N}_{i-1,1},\mathcal{N}_{i-1,2},\ldots,\mathcal{N}_{i-1,i-1}$ into two subsets.
	\end{enumerate}
	If $\mathcal{N}=(\mathcal{N}_1,\mathcal{N}_2,\ldots,\mathcal{N}_n)$ is a Soules-type sequence of partitions of $\{1,2,\ldots,${ }$n\}$, then we label the sets $\mathcal{N}_{i,k}$ and $\mathcal{N}_{i,l}$ in (ii) as $\mathcal{N}_i^*$ and $\mathcal{N}_i^{**}$, i.e. for $i\in\{2,3,\ldots,n\}$, we define $\mathcal{N}_i^*$ and $\mathcal{N}_i^{**}$ to be those sets in $\mathcal{N}_i$ which do not coincide with any of the sets in $\mathcal{N}_{i-1}$.
\end{defin}

\begin{defin}
Let $x\in\mathbb{R}^n$ be a positive vector and let $\mathcal{N}=(\mathcal{N}_1,\mathcal{N}_2,${ }$\ldots,\mathcal{N}_n)$ be a Soules-type sequence of partitions of $\{1,2,\ldots,n\}$. For each $i\in\{2,3,\ldots,n\}$, we define $x_\mathcal{N}^{(i)}$ to be the vector in $\mathbb{R}^n$ whose $i^\text{th}$ component is:
\[
	\left\{
		\begin{array}{ll}
			x_i & :\; i\in\mathcal{N}_i^*\\
			0 & :\; i\notin\mathcal{N}_i^*
		\end{array}
	\right.
\]
and we define $\hat{x}_\mathcal{N}^{(i)}$ to be the vector in $\mathbb{R}^n$ whose $i^\text{th}$ component is:
\[
	\left\{
		\begin{array}{ll}
			x_i & :\; i\in\mathcal{N}_i^{**}\\
			0 & :\; i\notin\mathcal{N}_i^{**}.
		\end{array}
	\right.
\]
\end{defin}

We are now ready to state the characterisation of Soules matrices due to Elsner, Nabben and Neumann:

\begin{thm}{\bf\cite{ElsnerNabbenNeumann}}\label{thm:SoulesMatricesCharacterised}
	Let $x\in\mathbb{R}^n$ be a positive vector and let $R$ be a Soules matrix with columns $r_1,r_2,\ldots,r_n$, where $r_1=x$. Then there exists a Soules-type sequence $\mathcal{N}$ of partitions of $\{1,2,\ldots,n\}$ such that $r_i$ is given (up to a factor of $\pm1$) by
	\begin{equation}\label{eq:SoulesColumns}
	 r_i=\frac{1}{\sqrt{||x_\mathcal{N}^{(i)}||_2^2+||\hat{x}_\mathcal{N}^{(i)}||_2^2}}\left( \frac{||\hat{x}_\mathcal{N}^{(i)}||_2}{||x_\mathcal{N}^{(i)}||_2}x_\mathcal{N}^{(i)}-\frac{||x_\mathcal{N}^{(i)}||_2}{||\hat{x}_\mathcal{N}^{(i)}||_2}\hat{x}_\mathcal{N}^{(i)} \right),
	\end{equation}
	$i=2,3,\ldots,n$.
	
	Conversely, if $x\in\mathbb{R}^n$ is a positive vector with $||x||_2=1$ and $\mathcal{N}$ is a Soules-type sequence of partitions of $\{1,2,\ldots,n\}$, then the matrix $R=\left[\begin{array}{cccc}r_1&r_2&\cdots&r_n \end{array}\right]$---with $r_1=x$ and $r_2,r_3,\ldots,r_n$ given by (\ref{eq:SoulesColumns})---is a Soules matrix.
\end{thm}

\begin{rem}
	Note that, by (\ref{eq:SoulesColumns}), the $j^\mathrm{th}$ entry of $r_i$ is nonzero if and only if $j\in\mathcal{N}_i^*\cup\mathcal{N}_i^{**}$.
\end{rem}

\begin{ex}\label{ex:SoulesExample}
	Let us show that $(7;5,-2,-4,-6)\in\mathcal{S}_5(0,0,0,0,0)$. To see this, conside the vector
	\[
		x=\left[\begin{array}{ccccc}\frac{1}{2}&\frac{1}{2}&\frac{1}{2 \sqrt{2}}&\frac{\sqrt{3}}{4}&\frac{\sqrt{3}}{4}\end{array}\right]^T
	\]
	and the partition sequence $\mathcal{N}=(\mathcal{N}_1,\mathcal{N}_2,\mathcal{N}_3,\mathcal{N}_4,\mathcal{N}_5)$ (illustrated in Figure \ref{fig:PartitionSequence}), where
	\begin{align*}
		\mathcal{N}_1 &=\{\{1,2,3,4,5\}\},\\
		\mathcal{N}_2 &=\{\{1,2\},\{3,4,5\}\},\\
		\mathcal{N}_3 &=\{\{1,2\},\{3\},\{4,5\}\},\\
		\mathcal{N}_4 &=\{\{1,2\},\{3\},\{4\},\{5\}\},\\
		\mathcal{N}_5 &=\{\{1\},\{2\},\{3\},\{4\},\{5\}\}.
	\end{align*}
	\begin{figure}[hbt]
		\centering
		\begin{tikzpicture}
[nodeDecorate/.style={shape=rectangle,inner sep=1pt,opacity=1}]
% nodes or vertices
\node (a) at (3,4) [nodeDecorate] {$\{1,2,3,4,5\}$};

\node (b) at (2.25,3) [nodeDecorate] {$\{1,2\}$};
\node (c) at (3.75,3) [nodeDecorate] {$\{3,4,5\}$};

\node (d) at (1.5,2) [nodeDecorate] {$\{1,2\}$};
\node (e) at (3,2) [nodeDecorate] {$\{3\}$};
\node (f) at (4.5,2) [nodeDecorate] {$\{4,5\}$};

\node (g) at (0.75,1) [nodeDecorate] {$\{1,2\}$};
\node (h) at (2.25,1) [nodeDecorate] {$\{3\}$};
\node (i) at (3.75,1) [nodeDecorate] {$\{4\}$};
\node (j) at (5.25,1) [nodeDecorate] {$\{5\}$};

\node (k) at (0,0) [style=nodeDecorate] {$\{1\}$};
\node (l) at (1.5,0) [style=nodeDecorate] {$\{2\}$};
\node (m) at (3,0) [nodeDecorate] {$\{3\}$};
\node (n) at (4.5,0) [nodeDecorate] {$\{4\}$};
\node (o) at (6,0) [nodeDecorate] {$\{5\}$};
% edges or lines
\tikzstyle{EdgeStyle}=[->,>=stealth,thick,opacity=1]
\tikzstyle{LabelStyle}=[fill=white]
\foreach \startnode/\endnode/\bend in {
  a/b/bend left=0,
  a/c/bend left=0,
  b/d/bend left=0,
  c/e/bend left=0,
  c/f/bend left=0,
  d/g/bend left=0,
  e/h/bend left=0,
  f/i/bend left=0,
  f/j/bend left=0,
  g/k/bend left=0,
  g/l/bend left=0,
  h/m/bend left=0,
  i/n/bend left=0,
  j/o/bend left=0
}
{
  \Edge[style=\bend](\startnode)(\endnode)
}
		\end{tikzpicture}
		\caption{Partition sequence $\mathcal{N}$}
		\label{fig:PartitionSequence}
	\end{figure}
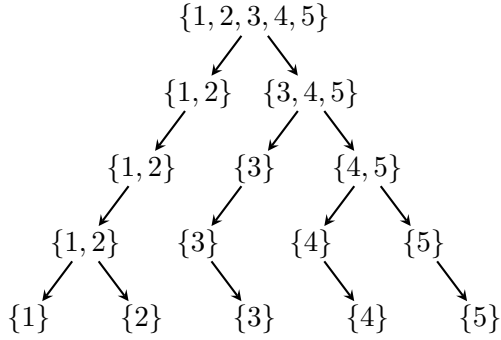
	Using (\ref{eq:SoulesColumns}), we construct the Soules matrix
	\[
		R=
\left[
\begin{array}{ccccc}
 \frac{1}{2} & \frac{1}{2} & 0 & 0 & \frac{1}{\sqrt{2}} \\
 \frac{1}{2} & \frac{1}{2} & 0 & 0 & -\frac{1}{\sqrt{2}} \\
 \frac{1}{2 \sqrt{2}} & -\frac{1}{2 \sqrt{2}} & \frac{\sqrt{3}}{2} & 0 & 0 \\
 \frac{\sqrt{3}}{4} & -\frac{\sqrt{3}}{4} & -\frac{1}{2 \sqrt{2}} & \frac{1}{\sqrt{2}} & 0 \\
 \frac{\sqrt{3}}{4} & -\frac{\sqrt{3}}{4} & -\frac{1}{2 \sqrt{2}} & -\frac{1}{\sqrt{2}} & 0
\end{array}
\right]
	\]
	and the realising matrix
	\begin{equation}\label{eq:RLA0}
		A=R\Lambda R^T=
\left[
\begin{array}{ccccc}
 0 & 6 & \frac{1}{2 \sqrt{2}} & \frac{\sqrt{3}}{4} & \frac{\sqrt{3}}{4} \\
 6 & 0 & \frac{1}{2 \sqrt{2}} & \frac{\sqrt{3}}{4} & \frac{\sqrt{3}}{4} \\
 \frac{1}{2 \sqrt{2}} & \frac{1}{2 \sqrt{2}} & 0 & \sqrt{6} & \sqrt{6} \\
 \frac{\sqrt{3}}{4} & \frac{\sqrt{3}}{4} & \sqrt{6} & 0 & 4 \\
 \frac{\sqrt{3}}{4} & \frac{\sqrt{3}}{4} & \sqrt{6} & 4 & 0
\end{array}
\right],
	\end{equation}
	where $\Lambda:=\mathrm{diag}(7,5,-2,-4,-6)$.
\end{ex}

Note that if $\sigma$ is irreducible, then any realising matrix for $\sigma$ has a positive Perron eigenvector; however, the condition that Soules matrices have positive first column means that certain reducible lists (which are trivially symmetrically realisable) are not contained in $\mathcal{S}_n$; for example, $(1,1)$ is symmetrically realisable, but $(1,1)\not\in\mathcal{S}_2$. In order to complete the equivalence we prove in Section \ref{sec:Equivalence}, we would like to include these reducible spectra in the Soules set. Hence we make the following definition:

\begin{defin}
	Let $\lambda_1\geq\lambda_2\geq\cdots\geq\lambda_n$ and let $a_1,a_2,\ldots,a_n\geq0$. We write
	\begin{equation}\label{eq:SBarDef}
		(\lambda_1;\lambda_2,\ldots,\lambda_n)\in\overline{\mathcal{S}}_n(a_1,a_2,\ldots,a_n)
	\end{equation}
	if there exist two partitions
	\begin{gather*}
		\{1,\ldots,n\}=\{\alpha_1^{(1)},\ldots,\alpha_{n_1}^{(1)}\}\cup\{\alpha_1^{(2)},\ldots,\alpha_{n_2}^{(2)}\}\cup\cdots\cup\{\alpha_1^{(k)},\ldots,\alpha_{n_k}^{(k)}\}, \\
		\{1,\ldots,n\}=\{\beta_1^{(1)},\ldots,\beta_{n_1}^{(1)}\}\cup\{\beta_1^{(2)},\ldots,\beta_{n_2}^{(2)}\}\cup\cdots\cup\{\beta_1^{(k)},\ldots,\beta_{n_k}^{(k)}\},
	\end{gather*}
	such that
	\[
		\left(\lambda_{\alpha_1^{(i)}};\lambda_{\alpha_2^{(i)}},\ldots,\lambda_{\alpha_{n_i}^{(i)}}\right)\in\mathcal{S}_{n_i}\left(a_{\beta_1^{(i)}},a_{\beta_2^{(i)}},\ldots,a_{\beta_{n_i}^{(i)}}\right) :\hspace{3mm} i=1,2,\ldots,k.
	\]
	We write
	\[
		(\lambda_1,\lambda_2,\ldots,\lambda_n)\in\overline{\mathcal{S}}_n
	\]
	if there exist $a_1,a_2,\ldots,a_n\geq0$ such that (\ref{eq:SBarDef}) holds.
\end{defin}

The Soules set and its role in the SNIEP has been extensively studied, for example by McDonald and Neumann \cite{McDonaldNeumann} and Loewy and McDonald \cite{LoewyMcDonald}. Soules matrices and the associated orthonormal bases have also been considered elsewhere in the literature, for example in \cite{CHN2006, CNS2007, Nabben2007, CNS2008, EubanksMcDonald2009}. In addition, Soules matrices have been applied to other areas of linear algebra, including nonnegative matrix factorisation \cite{CHNP2004}, the cp-rank problem \cite{Naomi2004} and describing the relationships between various classes of matrices \cite{ElsnerNabbenNeumann, Nabben2007}.

\subsection{A constructive lemma}

In \cite[Lemma 5]{SmigocDiagonalElement}, given a nonnegative matrix $B$ with Perron eigenvalue $c$ and specrtum $(c,\nu_2,\nu_3,\ldots,\nu_l)$ and a nonnegative matrix $A$ with spectrum $(\mu_1,\mu_2,\ldots,\mu_k)$ and a diagonal element $c$, \v{S}migoc shows how to construct a nonnegative matrix $C$ with spectrum $(\mu_1,${ }$\mu_2,\ldots,\mu_k,\nu_2,\nu_3,\ldots,\nu_l)$. For applications of this construction, see \cite{SmigocDiagonalElement,SmigocSubmatrixConstruction,NewListsFromOld}. Furthermore, if $A$ and $B$ are symmetric, then $C$ will be symmetric also. We state the symmetric case below.

\begin{lem}{\bf \cite{SmigocDiagonalElement}}\label{lem:SmigocSDLemma}
	Let $B$ be an $l\times l$ nonnegative symmetric matrix with Perron eigenvalue $c$ and spectrum $(c,\nu_2,\nu_3,\ldots,\nu_l)$ and let $Y\in\mathbb{R}^{l\times l}$ be an orthogonal matrix such that
	\[
		Y^TBY=\mathrm{diag}(c,\nu_2,\nu_3,\ldots,\nu_l).
	\]
	Let $Y$ be partitioned as
	\[
		Y=\left[
			\begin{array}{cc}
				v & V
			\end{array}
		\right]
	\]
	where $v\in\mathbb{R}^l$ and $V\in\mathbb{R}^{l\times(l-1)}$.
	
	Let
	\[
		A:=\left[
			\begin{array}{cc}
				A_1 & a \\
				a^T & c
			\end{array}
		\right],
	\]
	where $A_1$ is an $(k-1)\times(k-1)$ nonnegative symmetric matrix and $a\in\mathbb{R}^{k-1}$ is nonnegative and let $X\in\mathbb{R}^{k\times k}$ be an orthogonal matrix such that
	\[
		X^TAX=\mathrm{diag}(\mu_1,\mu_2,\ldots,\mu_k).
	\]
	Let $X$ be partitioned as
	\[
		X=\left[
			\begin{array}{c}
				U \\ u^T
			\end{array}
		\right],
	\]
	where $u\in\mathbb{R}^k$ and $U\in\mathbb{R}^{(k-1)\times k}$.	
	
	Then for matrices
	\[
		C:=\left[
			\begin{array}{cc}
				A_1 & av^T \\
				va^T & B
			\end{array}
		\right]	
	\]
	and
	\[
		Z:=\left[
			\begin{array}{cc}
				U & 0 \\
				vu^T & V
			\end{array}
		\right],
	\]
	we have
	\[
		Z^TCZ=\mathrm{diag}(\mu_1,\mu_2,\ldots,\mu_k,\nu_2,\nu_3,\ldots,\nu_l).
	\]
\end{lem}

\subsection{C-realisability and the RNIEP}

In \cite{UnifiedView}, Borobia, Moro and Soto construct realisable lists in the RNIEP, starting from trivially realisable lists, using three well-known results.

Specifically, in 1997, Guo gave the following result, which states that we may perturb a real eigenvalue of a realisable list by $\pm\epsilon$, provided we also increase the Perron eigenvalue by $\epsilon$:

\begin{thm}{\bf\cite{Guo}}\label{thm:Guo}
	If $(\rho,\lambda_2,\lambda_3,\ldots,\lambda_n)$ is realisable, where $\rho$ is the Perron
eigenvalue and $\lambda_2$ is real, then
	\[
		(\rho+\epsilon,\lambda_2\pm \epsilon,\lambda_3,\lambda_4,\ldots,\lambda_n)
	\]
	is realisable for all $\epsilon\geq0$.
\end{thm}

Note also the following well-known result, also proved by Guo \cite{Guo}:

\begin{thm}{\bf \cite{Guo}}\label{thm:PerronIncrease}
	If $(\rho,\lambda_2,\lambda_3,\ldots,\lambda_n)$ is the spectrum of a nonnegative matrix with Perron eigenvalue $\rho$, then for all $\epsilon\geq0$, $(\rho+\epsilon,\lambda_2,\lambda_3,\ldots,${ }$\lambda_n)$ is the spectrum of a nonnegative matrix also.
\end{thm}

Finally, recall that the spectrum of a block diagonal matrix is the union of the spectra of the diagonal blocks, in other words:

\begin{obv}\label{obv:Union}
	If $(\lambda_1,\lambda_2,\ldots,\lambda_m)$ and $(\mu_1,\mu_2,\ldots,\mu_n)$ are realisable, then $(\lambda_1,\lambda_2,\ldots,\lambda_m,\mu_1,\mu_2,\ldots,\mu_n)$ is realisable.
\end{obv}

Borobia, Moro and Soto make the following definition:

\begin{defin}
	A list of real numbers $(\lambda_1,\lambda_2,\ldots,\lambda_n)$ is called \emph{C-realisable} if it may be obtained by starting with the $n$ trivially realisable lists $(0),(0),${ }$\ldots,(0)$ and then using results \ref{thm:Guo}, \ref{thm:PerronIncrease} and \ref{obv:Union} any number of times in any order.
\end{defin}

\begin{ex}\label{ex:C-realisabilityExample}
	In Example \ref{ex:SoulesExample}, we showed that $(7,5,-2,-4,-6)\in\mathcal{S}_5$. To see that $(7,5,-2,-4,-6)$ is C-realisable, consider the following series of steps:
	\begin{enumerate}
		\item $(0),(0),(0),(0),(0)$
		\item $(0,0),(0,0),(0)$
		\item $(6,-6),(4,-4),(0)$
		\item $(6,-6),(4,0,-4)$
		\item $(6,-6),(6,-2,-4)$
		\item $(6,6,-2,-4,-6)$
		\item $(7,5,-2,-4,-6)$
	\end{enumerate}
	We used Ovservation \ref{obv:Union} at steps $1\rightarrow2$, $3\rightarrow4$ and $5\rightarrow6$. We used Theorem \ref{thm:Guo} at steps $2\rightarrow3$, $4\rightarrow5$ and $6\rightarrow7$.
\end{ex}

Of course, if $\sigma$ is C-realisable, then $\sigma$ is realisable. Note that while the symmetric analogues of Theorem \ref{thm:PerronIncrease} and Observation \ref{obv:Union} hold, it is an open question whether the symmetric version of Theorem \ref{thm:Guo} is true. We prove in Section \ref{sec:Equivalence} that if $\sigma$ is C-realisable, then $\sigma$ is symmetrically realisable.

\subsection{A family of realisability criteria in the SNIEP}\label{sec:Sp}

Based on a theorem of Brauer, Soto \cite{Soto2013} gives a family of realisability criteria denoted $\mathbb{S}_1,\mathbb{S}_2,\ldots$ (not to be confused with $\mathcal{S}_n$), such that if a list of real numbers $\sigma:=(\lambda_1,\lambda_2,\ldots,\lambda_n)$ satisfies the criterion $\mathbb{S}_p$ for some $p=1,2,\ldots$, then $\sigma$ is realisable. Soto also shows in \cite{Soto2013} that the $\mathbb{S}_p$ criteria are sufficient for symmetric realisability. In order to state $\mathbb{S}_p$, we will require some terminology and notation from \cite{Soto2013}: Let $\sigma:=(\lambda_1,\lambda_2,\ldots,\lambda_n)$, where $\lambda_1\geq\lambda_2\geq\cdots\geq\lambda_n$, and let $K$ be a realisability criterion. Then we write
\[
	\sigma\in\mathcal{Q}_K
\]
if $\sigma$ satisfies the criterion $K$. The \emph{Brauer} $K$-\emph{negativity} of $\sigma$ is defined to be the nonnegative number
\begin{equation}\label{eq:SotoNDef}
	\mathcal{N}_K(\sigma):=\min\{\epsilon\geq0:(\lambda_1+\epsilon,\lambda_2,\lambda_3,\ldots,\lambda_n)\in\mathcal{Q}_K\},
\end{equation}
and if $\sigma\in\mathcal{Q}_K$, then the \emph{Brauer} $K$-\emph{realisability margin} of $\sigma$ is defined to be
\begin{equation}\label{eq:SotoMDef}
	\mathcal{M}_K(\sigma):=\max\{\epsilon\in[0,\lambda_1-\lambda_2]:(\lambda_1-\epsilon,\lambda_2,\lambda_3,\ldots,\lambda_n)\in\mathcal{Q}_K\}.
\end{equation}

The $\mathbb{S}_p$ criteria are now defined recursively: We say $\sigma$ satisfies the $\mathbb{S}_1$ criterion if
\[
	\lambda_1\geq-\lambda_n-\sum_{T_i<0}T_i,
\]
where
\[
	T_i:=\lambda_i+\lambda_{n-i+1}: \hspace{3mm} i=2,3,\ldots,\left\lfloor\frac{n}{2}\right\rfloor
\]
and for odd $n\geq3$, $T_{\frac{n+1}{2}}:=\min\left\{\lambda_{\frac{n+1}{2}},0\right\}$.

For $p=2,3,\ldots$, we say that $\sigma$ satisfies the $\mathbb{S}_p$ criterion if there exists a partition of $\sigma$ into sublists $\sigma_1,\sigma_2,\ldots,\sigma_r$, where
\begin{equation}\label{eq:SotoPartition}
	\left\{\hspace{3mm}
		\begin{array}{ccl}
			\sigma_i=\left(\lambda^{(i)}_1,\lambda^{(i)}_2,\ldots,\lambda^{(i)}_{n_i}\right) & : & i=1,2,\ldots,r, \\
			\lambda^{(1)}_1=\lambda_1, & & \\
			\lambda^{(i)}_1\geq0 & : & i=1,2,\ldots,r, \\
			\lambda^{(i)}_1\geq\lambda^{(i)}_2\geq\cdots\geq\lambda^{(i)}_{n_i} & : & i=1,2,\ldots,r,
		\end{array}
	\right.
\end{equation}
such that $\sigma_1\in\mathcal{Q}_{\mathbb{S}_{p-1}}$ and
\begin{equation}\label{eq:Soto1}
	\lambda_1\geq\gamma+\sum_{\sigma_i\not\in\mathcal{Q}_{\mathbb{S}_{p-1}}}\mathcal{N}_{\mathbb{S}_{p-1}}(\sigma_i),
\end{equation}
where
\begin{equation}\label{eq:Soto2}
	\gamma:=\max\{\lambda_1-\mathcal{M}_{\mathbb{S}_{p-1}}(\sigma_1),\lambda^{(2)}_1,\lambda^{(3)}_1,\ldots,\lambda^{(r)}_1\}.
\end{equation}

Note that if we allow $r=1$ above, then we have:

\begin{obv}\label{obv:p_implies_p+1}
	If $\sigma$ satisfies $\mathbb{S}_p$, then $\sigma$ satisfies $\mathbb{S}_{p+1}$.
\end{obv}

\begin{thm}{\bf\cite{Soto2013}}
	If $\sigma$ satisfies $\mathbb{S}_p$ for any $p$, then $\sigma$ is symmetrically realisable.
\end{thm}

\begin{ex}\label{ex:SpExample}
	In Example \ref{ex:SoulesExample}, we showed that $\sigma:=(7,5,-2,-4,${ }$-6)\in\mathcal{S}_5$ and in Example \ref{ex:C-realisabilityExample}, we showed that $\sigma$ is C-realisable. It is easy to check that $\sigma$ does not satisfy $\mathbb{S}_1$; however, consider the partition $\sigma=(\sigma_1,\sigma_2)$, where $\sigma_1=(7,-6)$ and $\sigma_2=(5,-2,-4)$. Then $\mathcal{M}_{\mathbb{S}_1}(\sigma_1)=\mathcal{N}_{\mathbb{S}_1}(\sigma_2)=1$ and hence $\sigma$ satisfies $\mathbb{S}_2$.
\end{ex}

\section{A recursive approach to the SNIEP}\label{sec:Hn}

Here, we describe a method of recursively constructing symmetrically realisable lists, starting with lists of length 2 and repeatedly applying Lemma \ref{lem:SmigocSDLemma}. Formally, we define the set $\mathcal{H}_n$ in the following way:

\begin{defin}\label{def:HDefinition}
	For $a\geq0$, we write $(\lambda)\in\mathcal{H}_1(a)$ if $\lambda=a$. For $a_1,a_2\geq0$, we write $(\lambda_1;\lambda_2)\in\mathcal{H}_2(a_1,a_2)$ if $\lambda_1\geq\max\{a_1,a_2\}$ and $\lambda_1+\lambda_2=a_1+a_2$. For $a_1,a_2,\ldots,a_m\geq0$, we write
	\begin{equation}\label{eq:HDef}
		(\lambda_1;\lambda_2,\ldots,\lambda_m)\in\mathcal{H}_m(a_1,a_2,\ldots,a_m)
	\end{equation}
	if there exist two partitions
	\begin{gather*}
		\{2,3,\ldots,m\}=\{\alpha_1,\alpha_2,\ldots,\alpha_{k}\}\cup\{\beta_1,\beta_2,\ldots,\beta_{m-k-1}\} \\
		\{1,2,\ldots,m\}=\{\gamma_1,\gamma_2,\ldots,\gamma_k\}\cup\{\delta_1,\delta_2,\ldots,\delta_{m-k}\}
	\end{gather*}
	and a nonnegative number $c$ such that
	\[
		(\lambda_1;\lambda_{\alpha_1},\lambda_{\alpha_2},\ldots,\lambda_{\alpha_k})\in\mathcal{H}_{k+1}(a_{\gamma_1},a_{\gamma_2},\ldots,a_{\gamma_k},c)
	\]
	and
	\[
		(c;\lambda_{\beta_1},\lambda_{\beta_2},\ldots,\lambda_{\beta_{m-k-1}})\in\mathcal{H}_{m-k}(a_{\delta_1},a_{\delta_2},\ldots,a_{\delta_{m-k}}).
	\]
	We write
	\[
		(\lambda_1,\lambda_2,\ldots,\lambda_n)\in\mathcal{H}_n
	\]
	if there exist $a_1,a_2,\dots,a_n\geq0$ such that (\ref{eq:HDef}) holds.
\end{defin}

Note that by Lemma \ref{lem:SmigocSDLemma}, if $(\lambda_1;\lambda_2,\ldots,\lambda_n)\in\mathcal{H}_n(a_1,a_2,\ldots,a_n)$, then $(\lambda_1,\lambda_2,\ldots,\lambda_n)$ is the spectrum of a nonnegative symmetric matrix with diagonal elements $(a_1,a_2,\ldots,a_n)$.

\begin{ex}\label{ex:HExample}
	In Example \ref{ex:SoulesExample}, we showed that $\sigma:=(7,5,-2,-4,${ }$-6)\in\mathcal{S}_5$, in Example \ref{ex:C-realisabilityExample}, we showed that $\sigma$ is C-realisable and in Example \ref{ex:SpExample}, we showed that $\sigma$ satisfies $\mathbb{S}_2$. Let us now show that $\sigma\in\mathcal{H}_5$. To do this, we need only show how to progressively decompose $\sigma$ according to Definition \ref{def:HDefinition}. One such decomposition is given in Figure \ref{fig:SigmaDecomp}.
	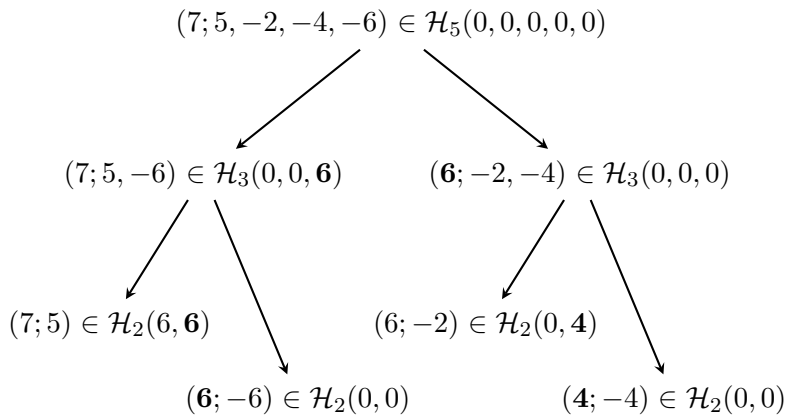
\begin{figure}[hbt]
		\centering
		\begin{tikzpicture}
[nodeDecorate/.style={shape=rectangle,inner sep=4pt,opacity=1}]
% nodes or vertices
\node (a) at (3.75,4) [nodeDecorate] {$(7;5,-2,-4,-6)\in\mathcal{H}_5(0,0,0,0,0)$};

\node (b) at (1.25,2) [nodeDecorate] {$(7;5,-6)\in\mathcal{H}_3(0,0,\mathbf{6})$};	
\node (c) at (6.25,2) [nodeDecorate] {$(\mathbf{6};-2,-4)\in\mathcal{H}_3(0,0,0)$};

\node (d) at (0,0) [nodeDecorate] {$(7;5)\in\mathcal{H}_2(6,\mathbf{6})$};	
\node (e) at (2.5,-1) [nodeDecorate] {$(\mathbf{6};-6)\in\mathcal{H}_2(0,0)$};
\node (f) at (5,0) [nodeDecorate] {$(6;-2)\in\mathcal{H}_2(0,\mathbf{4})$};	
\node (g) at (7.5,-1) [nodeDecorate] {$(\mathbf{4};-4)\in\mathcal{H}_2(0,0)$};

% edges or lines
\tikzstyle{EdgeStyle}=[->,>=stealth,thick,opacity=1]
\tikzstyle{LabelStyle}=[fill=white]
\foreach \startnode/\endnode/\bend in {
  a/b/bend left=0,
  a/c/bend left=0,
  b/d/bend left=0,
  b/e/bend left=0,
  c/f/bend left=0,
  c/g/bend left=0
}
{
  \Edge[style=\bend](\startnode)(\endnode)
}
		\end{tikzpicture}
		\caption{Decomposition of $\sigma$ into lists of length 2}
		\label{fig:SigmaDecomp}
	\end{figure}
\end{ex}

Note that the conditions given in (\ref{eq:n=2NS}) coincide with the definition of $\mathcal{H}_2(a_1,a_2)$ and hence $(\lambda_1,\lambda_2)$ is the spectrum of a nonnegative symmetric matrix with Perron eigenvalue $\lambda_1$ and diagonal elements $(a_1,a_2)$ if and only if $(\lambda_1;\lambda_2)\in\mathcal{H}_2(a_1,a_2)$. In the following lemma, we show that the same holds for $n=3$:

\begin{lem}\label{lem:n=3}
	Let $\lambda_1\geq\lambda_2\geq\lambda_3$ and $a_1\geq a_2\geq a_3\geq0$. If $(\lambda_1,\lambda_2,\lambda_3)$ is the spectrum of a nonnegative symmetric matrix with diagonal elements $(a_1,a_2,a_3)$, then
	\[
		(\lambda_1;\lambda_2)\in\mathcal{H}_2(a_1,c)
		\hspace{5mm}\mathrm{and}\hspace{5mm}
		(c;\lambda_3)\in\mathcal{H}_2(a_2,a_3),
	\]
	where $c:=\lambda_1+\lambda_2-a_1$. In particular, $(\lambda_1;\lambda_2,\lambda_3)\in\mathcal{H}_3(a_1,a_2,a_3)$.
\end{lem}
\begin{proof}
	The result follows easily from (\ref{eq:n=3NS}) and the definition of $\mathcal{H}_2$.
\end{proof}

Suppose that for all $n\in\{2,3,\ldots,m-1\}$ and $a_1,a_2,\ldots,a_{m-1}\geq0$, the sets $\mathcal{H}_n(a_1,a_2,\ldots,a_n)$ are known and we wish to determine whether $(\lambda_1;\lambda_2,\ldots,\lambda_m)\in\mathcal{H}_m(b_1,b_2,\ldots,b_m)$.
% From Definition \ref{def:HDefinition}, we would need to consider all $k\in\{1,2,\ldots,m-2\}$, all $k$-subsets of $\{2,3,\ldots,m\}$ and all $k$-subsets of $\{1,2,\ldots,m\}$: a total of
%\[
%	\sum_{k=1}^{m-2}\binom{m-1}{k}\binom{m}{k}=\binom{2m-1}{m-1}-m-1
%\]
%possibilities. Our next result reduces the number of possibilities to $\binom{m}{2}$ and shows that $\mathcal{H}_m$ depends only on $\mathcal{H}_{m-1}$ and $\mathcal{H}_2$:
Our next result shows that $\mathcal{H}_m$ depends only on $\mathcal{H}_{m-1}$ and $\mathcal{H}_2$:

\begin{thm}\label{thm:HSimplification}
	Let $n\geq3$, let $\lambda_1\geq\lambda_2\geq\cdots\geq\lambda_n$ and let $a_1,a_2,\ldots,a_n${ }$\geq0$. Then $(\lambda_1;\lambda_2,\ldots,\lambda_n)\in\mathcal{H}_n(a_1,a_2,\ldots,a_n)$ if and only if there exist $s,t\in\{1,2,\ldots,n\}$, $s<t$, such that
	\begin{multline}\label{eq:NEC1}
		(\lambda_1;\lambda_2,\ldots,\lambda_{n-1})\in \\ \mathcal{H}_{n-1}(a_1,\ldots,a_{s-1},a_{s+1},\ldots,a_{t-1},a_{t+1},\ldots,a_n,c)
	\end{multline}
	and
	\begin{equation}\label{eq:NEC2}
		(c;\lambda_n)\in\mathcal{H}_2(a_s,a_t),
	\end{equation}
	where $c:=a_s+a_t-\lambda_n$.
\end{thm}
\begin{proof}
	That (\ref{eq:NEC1}) and (\ref{eq:NEC2}) imply $(\lambda_1;\lambda_2,\ldots,\lambda_n)\in\mathcal{H}_n(a_1,a_2,\ldots,a_n)$ follows from Definition \ref{def:HDefinition}. Conversely, assume $(\lambda_1;\lambda_2,\ldots,\lambda_n)\in\mathcal{H}_n(a_1,a_2,\ldots,a_n)$. We claim that there exist $s$ and $t$ such that (\ref{eq:NEC1}) and (\ref{eq:NEC2}) hold.	
	
	We prove our claim by induction on $n$. If $n=3$, then the claim follows from Lemma \ref{lem:n=3}. Now assume the claim holds for all $n\in\{3,4,\ldots,m-1\}$, $m\geq4$, and suppose $(\lambda_1;\lambda_2,\ldots,\lambda_m)\in\mathcal{H}_m(a_1,a_2,${ }$\ldots,a_m)$. Then there exist a partition of $\{2,3,\ldots,m\}$ into two subsets $\{\alpha_1,\alpha_2,\ldots,\alpha_{k}\}$ and $\{\beta_1,\beta_2,\ldots,\beta_{m-k-1}\}$, a partition of $\{1,2,\ldots,m\}$ into two subsets $\{\gamma_1,\gamma_2,${ }$\ldots,\gamma_k\}$ and $\{\delta_1,\delta_2,\ldots,\delta_{m-k}\}$ and a nonnegative number $\hat{c}$ such that
	\begin{equation}\label{eq:FirstSplit1}
		(\lambda_1;\lambda_{\alpha_1},\lambda_{\alpha_2},\ldots,\lambda_{\alpha_k})\in\mathcal{H}_{k+1}(a_{\gamma_1},a_{\gamma_2},\ldots,a_{\gamma_k},\hat{c})
	\end{equation}
	and
	\begin{equation}\label{eq:FirstSplit2}
		(\hat{c};\lambda_{\beta_1},\lambda_{\beta_2},\ldots,\lambda_{\beta_{m-k-1}})\in\mathcal{H}_{m-k}(a_{\delta_1},a_{\delta_2},\ldots,a_{\delta_{m-k}}).
	\end{equation}
	We will show that this implies the existence of some $s$ and $t$ such that
	\begin{multline}\label{eq:NEC1m}
		(\lambda_1;\lambda_2,\ldots,\lambda_{m-1})\in \\ \mathcal{H}_{m-1}(a_1,\ldots,a_{s-1},a_{s+1},\ldots,a_{t-1},a_{t+1},\ldots,a_m,c)
	\end{multline}
	and
	\begin{equation}\label{eq:NEC2m}
		(c;\lambda_m)\in\mathcal{H}_2(a_s,a_t),
	\end{equation}
	where $c:=a_s+a_t-\lambda_m$.
	
	Without loss of generality, assume that the $\alpha_i$, $\beta_i$, $\gamma_i$ and $\delta_i$ are labelled so that $\alpha_1<\alpha_2<\cdots<\alpha_k$, $\beta_1<\beta_2<\cdots<\beta_{m-k-1}$, $\gamma_1<\gamma_2<\cdots<\gamma_k$ and $\delta_1<\delta_2<\cdots<\delta_{m-k}$. Since the $\lambda_i$ are ordered also, we must have either $\alpha_k=m$ or $\beta_{m-k-1}=m$. If $\alpha_k=m$, then we must distinguish the cases $k=1$ and $k>1$. If $\beta_{m-k-1}=m$, then we must distinguish the cases $k=m-2$ and $k<m-2$. In summary, we need to consider four possible cases:
	\begin{center}
		\begin{tabular}{lc}
			\hline
			\hspace{3mm} Case 1: & $k=m-2$, $\beta_1=m$ \\
			\hspace{3mm} Case 2: & $k<m-2$, \:$\beta_{m-k-1}=m$\hspace{3mm} \\
			\hspace{3mm} Case 3: & $k>1$, $\alpha_k=m$ \\
			\hspace{3mm} Case 4: & $k=1$, $\alpha_1=m$ \\
			\hline
		\end{tabular}
	\end{center}
	
	\underline{\emph{Case 1:}} There is nothing to prove.
	
	\underline{\emph{Case 2:}} Applying the inductive hypothesis to (\ref{eq:FirstSplit2}), there exist $\hat{s},\hat{t}\in\{1,2,\ldots,m-k\}$, $\hat{s}<\hat{t}$, such that
	\begin{multline}\label{eq:NEC3}
		(\hat{c};\lambda_{\beta_1},\lambda_{\beta_2},\ldots,\lambda_{\beta{m-k-2}})\in \\ \mathcal{H}_{m-k-1}(a_{\delta_1},\ldots,a_{\delta_{\hat{s}-1}},a_{\delta_{\hat{s}+1}},\ldots,a_{\delta_{\hat{t}-1}},a_{\delta_{\hat{t}+1}},\ldots,a_{\delta_{m-k}},c)
	\end{multline}
	and
	\begin{equation}\label{eq:NEC4}
		(c;\lambda_m)\in\mathcal{H}_2(a_{\delta_{\hat{s}}},a_{\delta{\hat{t}}}),
	\end{equation}
	where $c:=a_{\delta_{\hat{s}}}+a_{\delta{\hat{t}}}-\lambda_m$. Hence, if we let $s=\delta_{\hat{s}}$ and $t=\delta_{\hat{t}}$, then by Definition \ref{def:HDefinition}, (\ref{eq:FirstSplit1}) and (\ref{eq:NEC3}) imply (\ref{eq:NEC1m}), and (\ref{eq:NEC4}) becomes (\ref{eq:NEC2m}). Hence, in Case 2, we have completed the inductive step and established our claim. In the remainder of the proof, we will make frequent use of Definition \ref{def:HDefinition}.
	
	\underline{\emph{Case 3:}} Applying the inductive hypothesis to (\ref{eq:FirstSplit1}), we see that one of the following two sub-cases must hold:
	
	\underline{\emph{Case 3 (a):}} There exist $\hat{s},\hat{t}\in\{1,2,\ldots,k\}$, $\hat{s}<\hat{t}$, such that
	\begin{multline}\label{eq:NEC5}
		(\lambda_1;\lambda_{\alpha_1},\lambda_{\alpha_2},\ldots,\lambda_{\alpha_{k-1}})\in \\
		\mathcal{H}_k(a_{\gamma_1},\ldots,a_{\gamma_{\hat{s}-1}},a_{\gamma_{\hat{s}+1}},\ldots,a_{\gamma_{\hat{t}-1}},a_{\gamma_{\hat{t}+1}},\ldots,a_{\gamma_k},\hat{c},c)
	\end{multline}
	and
	\begin{equation}\label{eq:NEC6}
		(c;\lambda_m)\in\mathcal{H}_2(a_{\gamma_{\hat{s}}},a_{\gamma_{\hat{t}}}),
	\end{equation}
	where $c:=a_{\gamma_{\hat{s}}}+a_{\gamma_{\hat{t}}}-\lambda_m$. In this case, if we let $s=\gamma_{\hat{s}}$ and $t=\gamma_{\hat{t}}$, then (\ref{eq:NEC5}) and (\ref{eq:FirstSplit2}) give (\ref{eq:NEC1m}), and (\ref{eq:NEC6}) is becomes (\ref{eq:NEC2m}), which establishes the claim in Case 3 (a).
	
	\underline{\emph{Case 3 (b):}} There exists $h\in\{1,2,\ldots,k\}$ such that
	\begin{equation}\label{eq:NEC7}
		(\lambda_1;\lambda_{\alpha_1},\lambda_{\alpha_2},\ldots,\lambda_{\alpha_{k-1}})\in \\
		\mathcal{H}_k(a_{\gamma_1},\ldots,a_{\gamma_{h-1}},a_{\gamma_{h+1}},\ldots,a_{\gamma_k},c')
	\end{equation}
	and
	\begin{equation}\label{eq:NEC8}
		(c';\lambda_m)\in\mathcal{H}_2(a_{\gamma_h},\hat{c}),
	\end{equation}
	where $c':=a_{\gamma_h}+\hat{c}-\lambda_m$. In this case, (\ref{eq:NEC8}) and (\ref{eq:FirstSplit2}) imply
	\begin{equation}\label{eq:NEC9}
		(c';\lambda_{\beta_1},\lambda_{\beta_2},\ldots,\lambda_{\beta_{m-k-1}},\lambda_m)\in\mathcal{N}_{m-k+1}(a_{\delta_1},a_{\delta_2},\ldots,a_{\delta_{m-k}},a_{\gamma_h}).
	\end{equation}
	Form equations (\ref{eq:NEC7}) and (\ref{eq:NEC9}), we see that we have reduced the problem to Case 2. %Since we have already established the claim in Case 2, we have now established it in Case 3 (b) also.
	
	\underline{\emph{Case 4:}} In Case 4, if we set $h:=\gamma_1$, then (\ref{eq:FirstSplit1}) and (\ref{eq:FirstSplit2}) become
	\begin{equation}\label{eq:NEC10}
		(\lambda_1;\lambda_m)\in\mathcal{H}_2(a_h,\hat{c})
	\end{equation}
	and
	\begin{equation}\label{eq:NEC11}
		(\hat{c};\lambda_2,\lambda_3,\ldots,\lambda_{m-1})\in\mathcal{H}_{m-1}(a_1,\ldots,a_{h-1},a_{h+1},\ldots,a_m),
	\end{equation}
	respectively. Applying the inductive hypothesis to (\ref{eq:NEC11}), there exist $p,q\in\{1,2,\ldots,m\}\setminus\{h\}$, $p<q$, such that
	\begin{equation}\label{eq:NEC12}
		(\hat{c};\lambda_2,\lambda_3,\ldots,\lambda_{m-2})\in\mathcal{H}_{m-2}(a_{r_1},a_{r_2},\ldots,a_{r_{m-3}},c')
	\end{equation}
	and
	\begin{equation}\label{eq:NEC13}
		(c';\lambda_{m-1})\in\mathcal{H}_2(a_p,a_q),
	\end{equation}
	where $\{r_1,r_2,\ldots,r_{m-3}\}=\{1,2,\ldots,m\}\setminus\{p,q,h\}$ and $c':=a_p+a_q-\lambda_{m-1}$. Then, by (\ref{eq:NEC10}) and (\ref{eq:NEC12}),
	\begin{equation}\label{eq:NEC14}
		(\lambda_1;\lambda_2,\lambda_3,\ldots,\lambda_{m-2},\lambda_m)\in\mathcal{H}_{m-1}(a_{r_1},a_{r_2},\ldots,a_{r_{m-3}},a_h,c'),
	\end{equation}
	and now, examining (\ref{eq:NEC14}) and (\ref{eq:NEC13}), we see that we have reduced the problem to Case 3. This establishes our claim in Case 4 and completes the proof.
\end{proof}

Recall Observation \ref{obv:Union} and Theorems \ref{thm:PerronIncrease} and \ref{thm:Guo}, which are the foundation of the definition of C-realisability. In order to prove that the lists in $\mathcal{H}_n$ are precisely the C-realisable lists, we need to show that the analogues of these three theorems hold also for lists in $\mathcal{H}_n$. This is the focus of our next three results.

Firstly, observe that, trivially, $(\lambda_1;\mu_1)\in\mathcal{H}_2(\lambda_1,\mu_1)$. Therefore
	\[
		(\lambda_1;\lambda_2,\ldots,\lambda_m,\mu_1)\in\mathcal{H}_{m+1}(a_1,a_2,\ldots,a_m,\mu_1)
	\]
	and hence we have:

\begin{obv}\label{obv:UnionForSoules}
	If $(\lambda_1;\lambda_2,\ldots,\lambda_m)\in\mathcal{H}_m(a_1,a_2,\ldots,a_m)$, $(\mu_1;\mu_2,\ldots,\mu_n)${ }$\in\mathcal{H}_n(b_1,b_2,\ldots,b_n)$ and $\lambda_1\geq\mu_1$, then
	\[
		(\lambda_1;\lambda_2,\ldots,\lambda_m,\mu_1,\mu_2,\ldots,\mu_n)\in\mathcal{H}_{m+n}(a_1,a_2,\ldots,a_m,b_1,b_2,\ldots,b_n).
	\]
\end{obv}

By Observation \ref{obv:UnionForSoules}, $\mathcal{H}$ is closed under union. Mirroring our definitions of $\mathcal{S}$ and $\overline{\mathcal{S}}$, we would like to have a notion of ``irreducibility'' for lists in $\mathcal{H}_n$. Therefore, we make the following definition:

\begin{defin}
	We write
	\begin{equation}\label{eq:H*Def}
		(\lambda_1;\lambda_2,\ldots,\lambda_n)\in\mathcal{H}^*_n(a_1,a_2,\ldots,a_n)
	\end{equation}
	if $(\lambda_1;\lambda_2,\ldots,\lambda_n)\in\mathcal{H}_n(a_1,a_2,\ldots,a_n)$ and there \emph{do not exist} partitions
	\begin{gather*}
		\{1,2,\ldots,n\}=\{p_1,p_2,\ldots,p_l\}\cup\{q_1,q_2,\ldots,q_{n-l}\}, \\
		\{1,2,\ldots,n\}=\{r_1,r_2,\ldots,r_l\}\cup\{s_1,s_2,\ldots,s_{n-l}\},
	\end{gather*}
	such that
	\[
		(\lambda_{p_1};\lambda_{p_2},\ldots,\lambda_{p_l})\in\mathcal{H}_l(a_{r_1},a_{r_2},\ldots,a_{r_l})
	\]
	and
	\[
		(\lambda_{q_1};\lambda_{q_2},\ldots,\lambda_{q_{n-l}})\in\mathcal{H}_{n-l}(a_{s_1},a_{s_2},\ldots,a_{s_{n-l}}).
	\]
	We write
	\[
		(\lambda_1,\lambda_2,\ldots,\lambda_n)\in\mathcal{H}^*_n
	\]
	if there exist $a_1,a_2,\dots,a_n\geq0$ such that (\ref{eq:H*Def}) holds.
\end{defin}

We now give the analogue of Theorem \ref{thm:PerronIncrease} for lists in $\mathcal{H}_n$:

\begin{lem}\label{lem:SymmetricPerronLemma}
	If $(\rho;\lambda_2,\lambda_3,\ldots,\lambda_n)\in\mathcal{H}_n(a_1,a_2,\ldots,a_n)$, then for all $\epsilon\geq0$,
	\[
		(\rho+\epsilon;\lambda_2,\lambda_3,\ldots,\lambda_n)\in\mathcal{H}_n(a_1+\epsilon,a_2,a_3,\ldots,a_n).
	\]
\end{lem}
\begin{rem}
	Note that since the $a_i$ are unordered in Lemma \ref{lem:SymmetricPerronLemma}, $a_1$ can take the place of any of the diagonal elements $a_1,a_2,\ldots,a_n$.
\end{rem}
\begin{proof}[Proof of Lemma \ref{lem:SymmetricPerronLemma}]
	We proceed by induction on $n$. If $n=2$ and $(\rho;\lambda_2)\in\mathcal{H}_2(a_1,a_2)$, then $\rho\geq\max\{a_1,a_2\}$ and $\rho+\lambda_2=a_1+a_2$. Hence $\rho+\epsilon\geq\max\{a_1+\epsilon,a_2\}$ and $(\rho+\epsilon)+\lambda_2=(a_1+\epsilon)+a_2$. Therefore $(\rho+\epsilon;\lambda_2)\in\mathcal{H}_2(a_1+\epsilon,a_2)$.
	
	Now assume that the assertion holds for all $n\in\{2,3,\ldots,m-1\}$ and consider the case when $n=m$. If
	$
		(\rho;\lambda_2,\lambda_3,\ldots,\lambda_m)\in\mathcal{H}_m(a_1,a_2,\ldots,${ }$a_m),
	$
	then there exist a partition of $\{2,3,\ldots,m\}$ into two subsets $\{\alpha_1,${ }$\alpha_2,\ldots,\alpha_{k}\}$ and $\{\beta_1,\beta_2,\ldots,\beta_{m-k-1}\}$, a partition of $\{1,2,\ldots,m\}$ into two subsets $\{\gamma_1,\gamma_2,\ldots,\gamma_k\}$ and $\{\delta_1,\delta_2,\ldots,\delta_{m-k}\}$ and a nonnegative number $c$ such that
	\begin{equation}\label{eq:PerronLem1}
		(\rho;\lambda_{\alpha_1},\lambda_{\alpha_2},\ldots,\lambda_{\alpha_k})\in\mathcal{H}_{k+1}(a_{\gamma_1},a_{\gamma_2},\ldots,a_{\gamma_k},c)
	\end{equation}
	and
	\begin{equation}\label{eq:PerronLem3}
		(c;\lambda_{\beta_1},\lambda_{\beta_2},\ldots,\lambda_{\beta_{m-k-1}})\in\mathcal{H}_{m-k}(a_{\delta_1},a_{\delta_2},\ldots,a_{\delta_{m-k}}).
	\end{equation}
	We now distinguish two possible cases: the case when $1\in\{\gamma_1,\gamma_2,\ldots,${ }$\gamma_k\}$ and the case when $1\in\{\delta_1,\delta_2,\ldots,\delta_{m-k}\}$.
	
	Suppose $1\in\{\gamma_1,\gamma_2,\ldots,\gamma_k\}$. Without loss of generality, we may assume that $1=\gamma_1$. By (\ref{eq:PerronLem1}) and the inductive hypothesis, we have that
	\begin{equation}\label{eq:PerronLem2}
		(\rho+\epsilon;\lambda_{\alpha_1},\lambda_{\alpha_2},\ldots,\lambda_{\alpha_k})\in\mathcal{H}_{k+1}(a_1+\epsilon,a_{\gamma_2},a_{\gamma_3},\ldots,a_{\gamma_k},c)
	\end{equation}
	and hence by (\ref{eq:PerronLem2}) and (\ref{eq:PerronLem3}),
	\begin{equation}\label{eq:PerronLem4}
		(\rho+\epsilon;\lambda_2,\lambda_3,\ldots,\lambda_m)\in\mathcal{H}_m(a_1+\epsilon,a_2,a_3,\ldots,a_m).
	\end{equation}
	
	Now suppose $1\in\{\delta_1,\delta_2,\ldots,\delta_{m-k}\}$. Without loss of generality, we may assume that $1=\delta_1$. Applying the inductive hypothesis to (\ref{eq:PerronLem1}) and (\ref{eq:PerronLem3}), gives
	\[
		(\rho+\epsilon;\lambda_{\alpha_1},\lambda_{\alpha_2},\ldots,\lambda_{\alpha_k})\in\mathcal{H}_{k+1}(a_{\gamma_1},a_{\gamma_2},\ldots,a_{\gamma_k},c+\epsilon)
	\]
	and
	\[
		(c+\epsilon;\lambda_{\beta_1},\lambda_{\beta_2},\ldots,\lambda_{\beta_{m-k-1}})\in\mathcal{H}_{m-k}(a_1+\epsilon,a_{\delta_2},a_{\delta_3},\ldots,a_{\delta_{m-k}}),
	\]
	respectively. Hence (\ref{eq:PerronLem4}) holds, as before.
\end{proof}

Note that Lemma \ref{lem:SymmetricPerronLemma} is true in general: in \cite{LSSymmetric}, Laffey and \v{S}migoc show that if $(\rho,\lambda_2,\lambda_3,\ldots,\lambda_n)$ is the spectrum of an irreducible nonnegative symmetric matrix with diagonal elements $(a_1,a_2,\ldots,a_n)$, then for all $\epsilon\geq0$, $(\rho+\epsilon,\lambda_2,\lambda_3,\ldots,\lambda_n)$ is the spectrum of a nonnegative symmetric matrix with diagonal elements $(a_1+\epsilon,a_2,a_3,\ldots,a_n)$.

We now give the analogue of Theorem \ref{thm:Guo} for lists in $\mathcal{H}_n$:

\begin{thm}\label{thm:GuoForSoules}
	Suppose $(\rho;\lambda_2,\lambda_3,\ldots,\lambda_n)\in\mathcal{H}_n(a_1,a_2,\ldots,a_n)$ and $\epsilon\geq0$. Then
	\begin{enumerate}
		\item[\textup{(i)}] $(\rho+\epsilon;\lambda_2-\epsilon,\lambda_3,\lambda_4,\ldots,\lambda_n)\in\mathcal{H}_n(a_1,a_2,\ldots,a_n)$;
		\item[\textup{(ii)}] there exist $s,t\in\{1,2,\ldots,n\}$, $s<t$, such that
		\begin{multline*}
			(\rho+\epsilon,\lambda_2+\epsilon,\lambda_3,\lambda_4,\ldots,\lambda_n)\in\\ \mathcal{H}_n(a_1,\ldots,a_{s-1},a_s+\epsilon,a_{s+1},\ldots,a_{t-1},a_t+\epsilon,a_{t+1},\ldots,a_n).
		\end{multline*}
	\end{enumerate}
	In particular, if $(\rho,\lambda_2,\lambda_3,\ldots,\lambda_n)\in\mathcal{H}_n$, then $(\rho+\epsilon,\lambda_2\pm\epsilon,\lambda_3,\ldots,\lambda_n)\in\mathcal{H}_n$.
\end{thm}
\begin{rem}
	Since the $\lambda_i$ are unordered, $\lambda_2$ may take the place of any of the eigenvalues $\lambda_2,\lambda_3,\ldots,\lambda_n$.
\end{rem}
\begin{proof}[Proof of Theorem \ref{thm:GuoForSoules}]
	We first consider the case when $n=2$. Suppose $(\rho;\lambda_2)\in\mathcal{H}_2(a_1,a_2)$. Then $\rho\geq\max\{a_1,a_2\}$ and $\rho+\lambda_2=a_1+a_2$. Therefore $\rho+\epsilon\geq\max\{a_1,a_2\}$ and $(\rho+\epsilon)+(\lambda_2-\epsilon)=a_1+a_2$. Hence $(\rho+\epsilon;\lambda_2-\epsilon)\in\mathcal{H}_2(a_1,a_2)$. Similarly, $\rho+\epsilon\geq\max\{a_1+\epsilon,a_2+\epsilon\}$ and $(\rho+\epsilon)+(\lambda_2+\epsilon)=(a_1+\epsilon)+(a_2+\epsilon)$. Hence $(\rho+\epsilon;\lambda_2+\epsilon)\in\mathcal{H}_2(a_1+\epsilon,a_2+\epsilon)$.
	
	Now assume the statement holds for all $n\in\{2,3,\ldots,m-1\}$ and suppose $(\rho;\lambda_2,\lambda_3,\ldots,\lambda_m)\in\mathcal{H}_m(a_1,a_2,\ldots,a_m)$. Then there exist a partition of $\{2,3,\ldots,m\}$ into two subsets $\{\alpha_1,\alpha_2,\ldots,\alpha_{k}\}$ and $\{\beta_1,\beta_2,${ }$\ldots,\beta_{m-k-1}\}$, a partition of $\{1,2,\ldots,m\}$ into two subsets $\{\gamma_1,\gamma_2,\ldots,\gamma_k\}$ and $\{\delta_1,\delta_2,\ldots,${ }$\delta_{m-k}\}$ and a nonnegative number $c$ such that
	\begin{equation}\label{eq:GuoForSoules1}
		(\rho;\lambda_{\alpha_1},\lambda_{\alpha_2},\ldots,\lambda_{\alpha_k})\in\mathcal{H}_{k+1}(a_{\gamma_1},a_{\gamma_2},\ldots,a_{\gamma_k},c)
	\end{equation}
	and
	\begin{equation}\label{eq:GuoForSoules2}
		(c;\lambda_{\beta_1},\lambda_{\beta_2},\ldots,\lambda_{\beta_{m-k-1}})\in\mathcal{H}_{m-k}(a_{\delta_1},a_{\delta_2},\ldots,a_{\delta_{m-k}}).
	\end{equation}
	Assume also that the $\gamma_i$ and $\delta_i$ are labelled so that $\gamma_1\leq\gamma_2\leq\cdots\leq\gamma_k$ and $\delta_1\leq\delta_2\leq\cdots\leq\delta_{m-k}$. We must distinguish between two possible cases: $2\in\{\alpha_1,\alpha_2,\ldots,\alpha_k\}$ and $2\in\{\beta_1,\beta_2,\ldots,\beta_{m-k-1}\}$.
	
	\underline{\emph{Case 1:}} Suppose $2\in\{\alpha_1,\alpha_2,\ldots,\alpha_k\}$. Without loss of generality, assume $2=\alpha_1$. By (\ref{eq:GuoForSoules1}) and the inductive hypothesis,
	\begin{equation}\label{eq:GuoForSoules3}
		(\rho+\epsilon;\lambda_2-\epsilon,\lambda_{\alpha_2},\lambda_{\alpha_3},\ldots,\lambda_{\alpha_k})\in\mathcal{H}_{k+1}(a_{\gamma_1},a_{\gamma_2},\ldots,a_{\gamma_k},c)
	\end{equation}
	and hence by (\ref{eq:GuoForSoules3}) and (\ref{eq:GuoForSoules2}),
	\begin{equation}\label{eq:GuoForSoules4}
		(\rho+\epsilon;\lambda_2-\epsilon,\lambda_3,\lambda_4,\ldots,\lambda_m)\in\mathcal{H}_m(a_1,a_2,\ldots,a_m).
	\end{equation}
	This proves (i) in Case 1. The inductive hypothesis also guarantees that one of the following sub-cases holds:
	
	\underline{\emph{Case 1 (a):}} There exist $\hat{s},\hat{t}\in\{1,2,\ldots,k\}$, $\hat{s}<\hat{t}$, such that
	\begin{multline}\label{eq:GuoForSoules7}
		(\rho+\epsilon,\lambda_2+\epsilon,\lambda_{\alpha_2},\lambda_{\alpha_3},\ldots,\lambda_{\alpha_k})\in\mathcal{H}_{k+1} \\ (a_{\gamma_1},\ldots,a_{\gamma_{\hat{s}-1}},a_{\gamma_{\hat{s}}}+\epsilon,a_{\gamma_{\hat{s}+1}},\ldots,a_{\gamma_{\hat{t}-1}},a_{\gamma_{\hat{t}}}+\epsilon,a_{\gamma_{\hat{t}+1}},\ldots,a_{\gamma_k},c).
	\end{multline}
	In this case, by (\ref{eq:GuoForSoules7}) and (\ref{eq:GuoForSoules2}),
	\begin{multline}\label{eq:GuoForSoules8}
		(\rho+\epsilon,\lambda_2+\epsilon,\lambda_3,\lambda_4,\ldots,\lambda_m)\in\\ \mathcal{H}_m(a_1,\ldots,a_{s-1},a_s+\epsilon,a_{s+1},\ldots,a_{t-1},a_t+\epsilon,a_{t+1},\ldots,a_m),
	\end{multline}
	where $s=\gamma_{\hat{s}}$ and $t=\gamma_{\hat{t}}$. This proves (ii) in Case 1 (a).
	
	\underline{\emph{Case 1 (b):}} There exists $\hat{s}\in\{1,2,\ldots,k\}$, such that
	\begin{multline}\label{eq:GuoForSoules9}
		(\rho+\epsilon,\lambda_2+\epsilon,\lambda_{\alpha_2},\lambda_{\alpha_3},\ldots,\lambda_{\alpha_k})\in\\ \mathcal{H}_{k+1}(a_{\gamma_1},\ldots,a_{\gamma_{\hat{s}-1}},a_{\gamma_{\hat{s}}}+\epsilon,a_{\gamma_{\hat{s}+1}},\ldots,a_{\gamma_k},c+\epsilon).
	\end{multline}
	In this case, applying Lemma \ref{lem:SymmetricPerronLemma} to (\ref{eq:GuoForSoules2}), we have
	\begin{equation}\label{eq:GuoForSoules10}
		(c+\epsilon;\lambda_{\beta_1},\lambda_{\beta_2},\ldots,\lambda_{\beta_{m-k-1}})\in\mathcal{H}_{m-k}(a_{\delta_1}+\epsilon,a_{\delta_2},a_{\delta_3},\ldots,a_{\delta_{m-k}}).
	\end{equation}
	and hence (\ref{eq:GuoForSoules8}) follows from (\ref{eq:GuoForSoules9}) and (\ref{eq:GuoForSoules10}),
	where $s=\min\{\gamma_{\hat{s}},\delta_1\}$ and $t=\max\{\gamma_{\hat{s}},\delta_1\}$. This proves (ii) in Case 1 (b).
	
	\underline{\emph{Case 2:}} Suppose  $2\in\{\beta_1,\beta_2,\ldots,\beta_{m-k-1}\}$ and without loss of generality, assume $2=\beta_1$. By (\ref{eq:GuoForSoules1}) and Lemma \ref{lem:SymmetricPerronLemma},
	\begin{equation}\label{eq:GuoForSoules5}
		(\rho+\epsilon;\lambda_{\alpha_1},\lambda_{\alpha_2},\ldots,\lambda_{\alpha_k})\in\mathcal{H}_{k+1}(a_{\gamma_1},a_{\gamma_2},\ldots,a_{\gamma_k},c+\epsilon).
	\end{equation}
	Applying the inductive hypothesis to (\ref{eq:GuoForSoules2}),
	\begin{equation}\label{eq:GuoForSoules6}
		(c+\epsilon;\lambda_2-\epsilon,\lambda_{\beta_2},\lambda_{\beta_3},\ldots,\lambda_{\beta_{m-k-1}})\in\mathcal{H}_{m-k}(a_{\delta_1},a_{\delta_2},\ldots,a_{\delta_{m-k}}).
	\end{equation}
	and there exist $\hat{s},\hat{t}\in\{1,2,\ldots,m-k\}$, $\hat{s}<\hat{t}$, such that
	\begin{multline}\label{eq:GuoForSoules11}
		(c+\epsilon;\lambda_2+\epsilon,\lambda_{\beta_2},\lambda_{\beta_3},\ldots,\lambda_{\beta_{m-k-1}})\in\mathcal{H}_{m-k} \\ (a_{\delta_1},\ldots,a_{\delta_{\hat{s}-1}},a_{\delta_{\hat{s}}}+\epsilon,a_{\delta_{\hat{s}+1}},\ldots,a_{\delta_{\hat{t}-1}},a_{\delta_{\hat{t}}}+\epsilon,a_{\delta_{\hat{t}+1}},\ldots,a_{\delta_{m-k}}).
	\end{multline}
	Equation (\ref{eq:GuoForSoules4}) then follows from (\ref{eq:GuoForSoules5}) and (\ref{eq:GuoForSoules6}), which proves (i) in Case 2. Equation (\ref{eq:GuoForSoules8}) follows from (\ref{eq:GuoForSoules5}) and (\ref{eq:GuoForSoules11}),
	where $s=\delta_{\hat{s}}$ and $t=\delta_{\hat{t}}$. This proves (ii) in Case 2.
\end{proof}

In \cite{Guo}, Guo conjectured that the symmetric analogue of Theorem \ref{thm:Guo} holds, i.e. that if $\sigma:=(\rho,\lambda_2,\lambda_3,\ldots,\lambda_n)$ is symmetrically realisable, then $(\rho+\epsilon,\lambda_2\pm\epsilon,\lambda_3,\lambda_4,\ldots,\lambda_n)$ is symmetrically realisable also. Whether this conjecture is true remains an open question; however, Theorem \ref{thm:GuoForSoules} shows the conjecture holds when $\sigma\in\mathcal{H}_n$. In particular, it says that if $\sigma\in\mathcal{H}_n$, then we may always increase the spectral gap whilst preserving symmetric realisability. Next, we show that if $\sigma\in\mathcal{H}_n$, then it is also possible to \emph{decrease} the spectral gap in the following sense: If $(\lambda_1,\lambda_2,\ldots,\lambda_n)\in\mathcal{H}^*_n$, where $\lambda_1\geq\lambda_2\geq\cdots\geq\lambda_n$,  then there exists $0<\epsilon\leq\frac{1}{2}(\lambda_1-\lambda_2)$ such that $(\lambda_1-\epsilon,\lambda_2+\epsilon,\lambda_3,\lambda_4,\ldots,\lambda_n)\in\mathcal{H}_n\setminus\mathcal{H}^*_n$.

\begin{thm}\label{thm:AntiGuoForSoules}
	Let $\lambda_1\geq\lambda_2\geq\cdots\geq\lambda_n$ and $a_1,a_2,\ldots,a_n\geq0$. Then $(\lambda_1;\lambda_2,\ldots,\lambda_n)\in\mathcal{H}_n(a_1,a_2,\ldots,a_n)$ if and only if there exist
	\[
		0\leq\epsilon\leq\frac{1}{2}(\lambda_1-\lambda_2)
	\]
	and two partitions
	\begin{equation}\label{eq:SSPart}\def\arraystretch{1.3}
		\begin{array}{c}
			\{3,4,\ldots,n\}=\{p_1,p_2,\ldots,p_{l-1}\}\cup\{q_1,q_2,\ldots,q_{n-l-1}\}, \\
			\{1,2,\ldots,n\}=\{r_1,r_2,\ldots,r_l\}\cup\{s_1,s_2,\ldots,s_{n-l}\},
		\end{array}
	\end{equation}
	such that
	\begin{equation}\label{eq:SS1}
		(\lambda_1-\epsilon;\lambda_{p_1},\lambda_{p_2},\ldots,\lambda_{p_{l-1}})\in\mathcal{H}_l(a_{r_1},a_{r_2},\ldots,a_{r_l})
	\end{equation}
	and
	\begin{equation}\label{eq:SS2}
		(\lambda_2+\epsilon;\lambda_{q_1},\lambda_{q_2},\ldots,\lambda_{q_{n-l-1}})\in\mathcal{H}_{n-l}(a_{s_1},a_{s_2},\ldots,a_{s_{n-l}}).
	\end{equation}
	We allow the possibilities $l=1$, in which case $\{p_1,p_2,\ldots,p_{l-1}\}$ is the empty set, and $l=n-1$, in which case $\{q_1,q_2,\ldots,q_{n-l-1}\}$ is the empty set.
\end{thm}
\begin{proof}
	First suppose there exist $\epsilon\in[0,\frac{1}{2}(\lambda_1-\lambda_2)]$ and partitions of the form (\ref{eq:SSPart}) such that (\ref{eq:SS1}) and (\ref{eq:SS2}) hold. Then by Observation \ref{obv:UnionForSoules},
	\[
		(\lambda_1-\epsilon;\lambda_2+\epsilon,\lambda_3,\lambda_4\ldots,\lambda_n)\in\mathcal{H}_n(a_1,a_2,\ldots,a_n)
	\]
	and hence by Theorem \ref{thm:GuoForSoules}, $(\lambda_1;\lambda_2,\ldots,\lambda_n)\in\mathcal{H}_n(a_1,a_2,\ldots,a_n)$. We claim the converse holds also.

	If $n=2$, then for $\epsilon=\lambda_1-a_1$, we have $\lambda_1-\epsilon=a_1$ and $\lambda_2+\epsilon=a_2$. Hence $(\lambda_1-\epsilon)\in\mathcal{H}_1(a_1)$ and $(\lambda_2+\epsilon)\in\mathcal{H}_1(a_2)$ and so the claim holds in this case. Now assume the claim holds for $n=m-1$ and consider the case when $n=m$.
	
	Suppose $(\lambda_1;\lambda_2,\ldots,\lambda_m)\in\mathcal{H}_m(a_1,a_2,\ldots,a_m)$. Then by Theorem \ref{thm:HSimplification}, there exist $s,t\in\{1,2,\ldots,m\}$, $s<t$, such that
	\begin{multline}\label{eq:SpectralGapReduction1}
		(\lambda_1;\lambda_2,\ldots,\lambda_{m-1})\in \\ \mathcal{H}_{m-1}(a_1,\ldots,a_{s-1},a_{s+1},\ldots,a_{t-1},a_{t+1},\ldots,a_m,c)
	\end{multline}
	and
	\begin{equation}\label{eq:SpectralGapReduction2}
		(c;\lambda_m)\in\mathcal{H}_2(a_s,a_t),
	\end{equation}
	where $c:=a_s+a_t-\lambda_m$. Let us now apply the inductive hypothesis to (\ref{eq:SpectralGapReduction1}); we will need to distinguish between the two possible cases $c\in\{r_1,r_2,\ldots,r_l\}$ and $c\in\{s_1,s_2,\ldots,s_{m-l-1}\}$.
	
	\underline{\emph{Case 1:}} There exist $\epsilon\in[0,\frac{1}{2}(\lambda_1-\lambda_2)]$ and two partitions
	\begin{gather*}
		\{3,4,\ldots,m-1\}=\{\hat{p}_1,\hat{p}_2,\ldots,\hat{p}_{\hat{l}-1}\}\cup\{q_1,q_2,\ldots,q_{m-\hat{l}-2}\}, \\
		\{1,2,\ldots,m\}\setminus\{s,t\}=\{\hat{r}_1,\hat{r}_2,\ldots,\hat{r}_{\hat{l}-1}\}\cup\{s_1,s_2,\ldots,s_{m-\hat{l}-1}\},
	\end{gather*}
	such that
	\begin{equation}\label{eq:SpectralGapReduction3}
		(\lambda_1-\epsilon;\lambda_{\hat{p}_1},\lambda_{\hat{p}_2},\ldots,\lambda_{\hat{p}_{\hat{l}-1}})\in\mathcal{H}_{\hat{l}}(a_{\hat{r}_1},a_{\hat{r}_2},\ldots,a_{\hat{r}_{\hat{l}-1}},c)
	\end{equation}
	and
	\begin{equation}\label{eq:SpectralGapReduction4}
		(\lambda_2+\epsilon;\lambda_{q_1},\lambda_{q_2},\ldots,\lambda_{q_{m-\hat{l}-2}})\in\mathcal{H}_{m-\hat{l}-1}(a_{s_1},a_{s_2},\ldots,a_{s_{m-\hat{l}-1}}).
	\end{equation}
	In this case, by Definition \ref{def:HDefinition}, (\ref{eq:SpectralGapReduction3}) and (\ref{eq:SpectralGapReduction2}) imply
	\begin{equation}\label{eq:SpectralGapReduction5}
		(\lambda_1-\epsilon;\lambda_{\hat{p}_1},\lambda_{\hat{p}_2},\ldots,\lambda_{\hat{p}_{\hat{l}-1}},\lambda_m)\in\mathcal{H}_{\hat{l}+1}(a_{\hat{r}_1},a_{\hat{r}_2},\ldots,a_{\hat{r}_{\hat{l}-1}},a_s,a_t)
	\end{equation}
	and after some relabelling, (\ref{eq:SpectralGapReduction5}) and (\ref{eq:SpectralGapReduction4}) become
	\[
		(\lambda_1-\epsilon;\lambda_{p_1},\lambda_{p_2},\ldots,\lambda_{p_{l-1}})\in\mathcal{H}_l(a_{r_1},a_{r_2},\ldots,a_{r_l})
	\]
	and
	\[
		(\lambda_2+\epsilon;\lambda_{q_1},\lambda_{q_2},\ldots,\lambda_{q_{m-l-1}})\in\mathcal{H}_{m-l}(a_{s_1},a_{s_2},\ldots,a_{s_{m-l}}),
	\]
	respectively. This completes the inductive step and establishes the claim in Case 1.
	
	\underline{\emph{Case 2:}} There exist $\epsilon\in[0,\frac{1}{2}(\lambda_1-\lambda_2)]$ and two partitions
	\begin{gather*}
		\{3,4,\ldots,m-1\}=\{p_1,p_2,\ldots,p_{l-1}\}\cup\{\hat{q}_1,\hat{q}_2,\ldots,\hat{q}_{m-l-2}\}, \\
		\{1,2,\ldots,m\}\setminus\{s,t\}=\{r_1,r_2,\ldots,r_l\}\cup\{\hat{s}_1,\hat{s}_2,\ldots,\hat{s}_{m-l-2}\},
	\end{gather*}
	such that
	\[
		(\lambda_1-\epsilon;\lambda_{p_1},\lambda_{p_2},\ldots,\lambda_{p_{l-1}})\in\mathcal{H}_l(a_{r_1},a_{r_2},\ldots,a_{r_l})
	\]
	and
	\begin{equation}\label{eq:SpectralGapReduction6}
		(\lambda_2+\epsilon;\lambda_{\hat{q}_1},\lambda_{\hat{q}_2},\ldots,\lambda_{\hat{q}_{m-l-2}})\in\mathcal{H}_{m-l-1}(a_{\hat{s}_1},a_{\hat{s}_2},\ldots,a_{\hat{s}_{m-l-2}},c).
	\end{equation}
	In this case, we may apply Definition \ref{def:HDefinition} to (\ref{eq:SpectralGapReduction6}) and (\ref{eq:SpectralGapReduction2}) and the remainder of the proof is analogous to Case 1.
\end{proof}

We finish this section by proving a property of $\mathcal{H}_n$ which will not have an application in proving our main equivalence result, but which is of independent interest. The effect of adding zeros to a list $\sigma$ has been extensively studied in the NIEP (see, for example, \cite{BoyleHandelman} and \cite{LaffeyConstructive}). The effect of adding zeros has also been studied in the SNIEP (see \cite{RealSymmetricDifferent}). Our next result shows that adding zeros to $\sigma$ does not affect whether $\sigma\in\mathcal{H}_n$:

\begin{thm}\label{thm:ZerosDon'tHelp}
	If
	\[
		(\lambda_1;\lambda_2,\ldots,\lambda_n,0)\in\mathcal{H}_{n+1}(a_1,a_2,\ldots,a_{n+1}),
	\]
	then there exist $s,t\in\{1,2,\ldots,n+1\}$, $s<t$, such that
	\begin{multline*}
		(\lambda_1;\lambda_2,\ldots,\lambda_n)\in \\ \mathcal{H}_n(a_1,\ldots,a_{s-1},a_{s+1},\ldots,a_{t-1},a_{t+1},\ldots,a_{n+1},a_s+a_t).
	\end{multline*}
	In particular, if $(\lambda_1,\lambda_2,\ldots,\lambda_n,0)\in\mathcal{H}_{n+1}$, then $(\lambda_1,\lambda_2,\ldots,\lambda_n)\in\mathcal{H}_n$.
\end{thm}
\begin{proof}
	If $n=1$, then the conclusion follows trivially. Now suppose $n=2$ and let
	\[
		(\lambda_1;\lambda_2,0)\in\mathcal{H}_3(a_1,a_2,a_3),
	\]
	where $a_1\geq a_2\geq a_3\geq0$. Then from the conditions given in (\ref{eq:n=3NS}),
	\begin{gather}
		\lambda_1\geq a_1,\label{eq:ZeroProp1}\\
		\lambda_2\leq a_1,\label{eq:ZeroProp2}
	\end{gather}
	and
	\begin{equation}\label{eq:ZeroProp3}
		\lambda_1+\lambda_2=a_1+a_2+a_3.
	\end{equation}
	Combining (\ref{eq:ZeroProp2}) and (\ref{eq:ZeroProp3}),
	\begin{equation}\label{eq:ZeroProp4}
		\lambda_1\geq a_2+a_3
	\end{equation}
	and hence from (\ref{eq:ZeroProp1}), (\ref{eq:ZeroProp4}) and (\ref{eq:ZeroProp3}), we see that
	\[
		(\lambda_1;\lambda_2)\in\mathcal{H}_2(a_1,a_2+a_3).
	\]
	
	Now assume the assertion holds for all $n\in\{1,2,\ldots,m-1\}$ and consider the case when $n=m$. If
	\[
		(\lambda_1;\lambda_2,\ldots,\lambda_m,0)\in\mathcal{H}_{m+1}(a_1,a_2,\ldots,a_{m+1}),
	\]
	then by Definition \ref{def:HDefinition}, one of the following cases must hold:
	
	\underline{\emph{Case 1:}} There exist a partition of $\{2,3,\ldots,m\}$ into two subsets $\{\alpha_1,${ }$\alpha_2,\ldots,\alpha_{k}\}$ and $\{\beta_1,\beta_2,\ldots,\beta_{m-k-1}\}$, a partition of $\{1,2,\ldots,m+1\}$ into two subsets $\{\gamma_1,\gamma_2,\ldots,\gamma_k\}$ and $\{\delta_1,\delta_2,\ldots,\delta_{m-k+1}\}$ and a nonnegative number $c$, such that
	\begin{equation}\label{eq:ZeroProp5}
		(\lambda_1;\lambda_{\alpha_1},\lambda_{\alpha_2},\ldots,\lambda_{\alpha_k})\in\mathcal{H}_{k+1}(a_{\gamma_1},a_{\gamma_2},\ldots,a_{\gamma_k},c)
	\end{equation}
	and
	\begin{equation}\label{eq:ZeroProp6}
		(c;\lambda_{\beta_1},\lambda_{\beta_2},\ldots,\lambda_{\beta_{m-k-1}},0)\in\mathcal{H}_{m-k+1}(a_{\delta_1},a_{\delta_2},\ldots,a_{\delta_{m-k+1}}).
	\end{equation}
	In this case, applying the inductive hypothesis to (\ref{eq:ZeroProp6}), we see that there exist $\hat{s},\hat{t}\in\{1,2,\ldots,m-k+1\}$, $\hat{s}<\hat{t}$, such that
	\begin{multline}\label{eq:ZeroProp7}
		(c;\lambda_{\beta_1},\lambda_{\beta_2},\ldots,\lambda_{\beta_{m-k-1}})\in \\ \mathcal{H}_{m-k}(a_{\delta_1},\ldots,a_{\delta_{\hat{s}-1}},a_{\delta_{\hat{s}+1}},\ldots,a_{\delta_{\hat{t}-1}},a_{\delta_{\hat{t}+1}},\ldots,a_{\delta_{m-k+1}},a_{\delta_{\hat{s}}}+a_{\delta_{\hat{t}}})
	\end{multline}
	and hence, assuming the $\delta_i$ are labelled so that $\delta_1<\delta_2<\cdots<\delta_{m-k+1}$, (\ref{eq:ZeroProp5}), (\ref{eq:ZeroProp7}) and Definition \ref{def:HDefinition} imply
	\begin{multline}\label{eq:ZerosDon'tHelpConclusion}
		(\lambda_1;\lambda_2,\ldots,\lambda_m)\in \\ \mathcal{H}_n(a_1,\ldots,a_{s-1},a_{s+1},\ldots,a_{t-1},a_{t+1},\ldots,a_{n+1},a_s+a_t),
	\end{multline}
	where $s=\delta_{\hat{s}}$ and $t=\delta_{\hat{t}}$. This completes the inductive step in Case 1.
	
	\underline{\emph{Case 2:}} There exist a partition of $\{2,3,\ldots,m\}$ into two subsets $\{\alpha_1,${ }$\alpha_2,\ldots,\alpha_{k}\}$ (which may be empty) and $\{\beta_1,\beta_2,\ldots,\beta_{m-k-1}\}$, a partition of $\{1,2,\ldots,m+1\}$ into two subsets $\{\gamma_1,\gamma_2,\ldots,\gamma_{k+1}\}$ and $\{\delta_1,\delta_2,${ }$\ldots,\delta_{m-k}\}$ and a nonnegative number $c$, such that
	\begin{equation}\label{eq:ZeroProp8}
		(\lambda_1;\lambda_{\alpha_1},\lambda_{\alpha_2},\ldots,\lambda_{\alpha_k},0)\in\mathcal{H}_{k+2}(a_{\gamma_1},a_{\gamma_2},\ldots,a_{\gamma_{k+1}},c)
	\end{equation}
	and
	\begin{equation}\label{eq:ZeroProp9}
		(c;\lambda_{\beta_1},\lambda_{\beta_2},\ldots,\lambda_{\beta_{m-k-1}})\in\mathcal{H}_{m-k}(a_{\delta_1},a_{\delta_2},\ldots,a_{\delta_{m-k}}).
	\end{equation}
	If $\{\alpha_1,\alpha_2,\ldots,\alpha_{k}\}$ is empty, then (\ref{eq:ZeroProp8}) reduces to $\lambda_1=a_{\gamma_1}+c$ and so, applying Lemma \ref{lem:SymmetricPerronLemma} to (\ref{eq:ZeroProp9}) with $\epsilon=a_{\gamma_1}$, gives
	\[
		(\lambda_1;\lambda_{\beta_1},\lambda_{\beta_2},\ldots,\lambda_{\beta_{m-k-1}})\in\mathcal{H}_{m-k}(a_{\delta_1}+a_{\gamma_1},a_{\delta_2},a_{\delta_3},\ldots,a_{\delta_{m-k}}),
	\]
	i.e. (\ref{eq:ZerosDon'tHelpConclusion}) holds with $s=\min\{\delta_1,\gamma_1\}$ and $t=\max\{\delta_1,\gamma_1\}$. If $\{\alpha_1,\alpha_2,\ldots,${ }$\alpha_{k}\}$ is nonempty, then applying the inductive hypothesis to (\ref{eq:ZeroProp8}) yields one of two possible sub-cases:
	
	\underline{\emph{Case 2 (a):}} There exist $\hat{s},\hat{t}\in\{1,2,\ldots,k+1\}$, $\hat{s}<\hat{t}$, such that
	\begin{multline}\label{eq:ZeroProp10}
		(\lambda_1;\lambda_{\alpha_1},\lambda_{\alpha_2},\ldots,\lambda_{\alpha_k})\in \\ \mathcal{H}_{k+1}(a_{\gamma_1},\ldots,a_{\gamma_{\hat{s}-1}},a_{\gamma_{\hat{s}+1}},\ldots,a_{\gamma_{\hat{t}-1}},a_{\gamma_{\hat{t}+1}},\ldots,a_{\gamma_{k+1}},a_{\gamma_{\hat{s}}}+a_{\gamma_{\hat{t}}},c)
	\end{multline}
	Then, assuming the $\gamma_i$ are labelled so that $\gamma_1<\gamma_2<\cdots<\gamma_{k+1}$, (\ref{eq:ZerosDon'tHelpConclusion}) holds by (\ref{eq:ZeroProp10}) and (\ref{eq:ZeroProp9}), with $s=\gamma_{\hat{s}}$ and $t=\gamma_{\hat{t}}$.
	
	\underline{\emph{Case 2 (b):}} There exists $\hat{r}\in\{1,2,\ldots,k+1\}$, such that
	\begin{multline}\label{eq:ZeroProp11}
		(\lambda_1;\lambda_{\alpha_1},\lambda_{\alpha_2},\ldots,\lambda_{\alpha_k})\in \\ \mathcal{H}_{k+1}(a_{\gamma_1},\ldots,a_{\gamma_{\hat{r}-1}},a_{\gamma_{\hat{r}+1}},\ldots,a_{\gamma_{k+1}},c+a_{\gamma_{\hat{r}}}).
	\end{multline}
	Then by applying Lemma \ref{lem:SymmetricPerronLemma} to (\ref{eq:ZeroProp9}) with $\epsilon=a_{\gamma_{\hat{r}}}$, we have that
	\begin{equation}\label{eq:ZeroProp12}
		(c+a_{\gamma_{\hat{r}}};\lambda_{\beta_1},\lambda_{\beta_2},\ldots,\lambda_{\beta_{m-k-1}})\in\mathcal{H}_{m-k}(a_{\delta_1}+a_{\gamma_{\hat{r}}},a_{\delta_2},a_{\delta_3},\ldots,a_{\delta_{m-k}})
	\end{equation}
	and hence (\ref{eq:ZeroProp11}) and (\ref{eq:ZeroProp12}) imply (\ref{eq:ZerosDon'tHelpConclusion}), where $s=\min\{\delta_1,\gamma_{\hat{r}}\}$ and $t=\max\{\delta_1,${ }$\gamma_{\hat{r}}\}$. This completes the inductive step in Case 2 and finishes the proof.
\end{proof}

\section{Main (equivalence) result}\label{sec:Equivalence}

We are now ready to prove the main result of this paper.

\begin{thm}\label{thm:SoulesEquivalence}
	The following are equivalent:
	\begin{enumerate}
		\item[\textup{(i)}] $(\lambda_1,\lambda_2,\ldots,\lambda_n)\in\overline{\mathcal{S}}_n$;
		\item[\textup{(ii)}] $(\lambda_1,\lambda_2,\ldots,\lambda_n)\in\mathcal{H}_n$;
		\item[\textup{(iii)}] $(\lambda_1,\lambda_2,\ldots,\lambda_n)$ is C-realisable;
		\item[\textup{(iv)}] $(\lambda_1,\lambda_2,\ldots,\lambda_n)$ satisfies $\mathbb{S}_p$ for some $p$.
	\end{enumerate}
	Furthermore, if $a_1,a_2,\ldots,a_n\geq0$ are given, then $(\lambda_1;\lambda_2,\ldots,\lambda_n)\in\overline{\mathcal{S}}_n(a_1,a_2,${ }$\ldots,a_n)$ if and only if $(\lambda_1;\lambda_2,\ldots,\lambda_n)\in\mathcal{H}_n(a_1,a_2,\ldots,a_n)$.
\end{thm}
\begin{proof}
	The proof is divided into six parts: In part 1, we show that $\mathcal{S}_n(a_1,a_2,${ }$\ldots,a_n)\subseteq\mathcal{H}_n(a_1,a_2,\ldots,a_n)$. In Part 2, we show that $\mathcal{H}^*_n(${ }$a_1,a_2,\ldots,a_n)\subseteq\mathcal{S}_n(a_1,a_2,\ldots,a_n)$. In Part 3, we use parts 1 and 2 to show that $\overline{\mathcal{S}}_n(a_1,a_2,\ldots,${ }$a_n)=\mathcal{H}_n(a_1,a_2,\ldots,a_n)$. In Part 4, we show that (ii) implies (iv). In Part 5, we show that (iv) implies (iii) and finally, in Part 6, we show that (iii) mplies (ii).

	\underline{\textsc{Part 1}:} Firstly, we claim that $\mathcal{S}_n(a_1,a_2,\ldots,a_n)\subseteq\mathcal{H}_n(a_1,a_2,\ldots,${ }$a_n)$. If $n=1$, there is nothing to prove. If $(\lambda_1;\lambda_2)\in\mathcal{S}_2(a_1,a_2)$, then in particular, $(\lambda_1,\lambda_2)$ is the spectrum of a nonnegative symmetric matrix with diagonal elements $(a_1,a_2)$. Therefore the conditions given in (\ref{eq:n=2NS}) hold and hence $(\lambda_1;\lambda_2)\in\mathcal{H}_2(a_1,a_2)$. Now assume the claim holds for all $n\in\{1,2,\ldots,m-1\}$, $m\geq3$, and consider the case when $n=m$.
	
	Suppose $(\lambda_1;\lambda_2,\ldots,\lambda_m)\in\mathcal{S}_m(a_1,a_2,\ldots,a_m)$, where $\lambda_1\geq\lambda_2\geq\cdots\geq\lambda_m$. Then there exists an $m\times m$ Soules matrix $R=(r_{ij})$ such that the matrix $R\Lambda R^T$---where  $\Lambda:=\mathrm{diag}(\lambda_1,\lambda_2,\ldots,\lambda_m)$---has diagonal elements $(a_1,a_2,\ldots,a_m)$. Assume also that the $a_i$ are labelled so that $a_j$ is the $(j,j)$ entry of $R\Lambda R^T$ and let us label the columns of $R$ as $x=r_1,r_2,\ldots,r_m$. Our aim is to construct two smaller Soules matrices $R_1$ and $R_2$ from $R$ and then apply the inductive hypothesis.
	
	By Theorem \ref{thm:SoulesMatricesCharacterised}, there exists a Soules-type sequence $\mathcal{N}=(\mathcal{N}_1,\mathcal{N}_2,${ }$\ldots,\mathcal{N}_m)$ of partitions of $\{1,2,\ldots,m\}$ such that for each $i\in\{2,3,\ldots,m\}$, $r_i$ is given (up to a factor of $\pm1$) by (\ref{eq:SoulesColumns}). As before, let us write $\mathcal{N}_i=\{\mathcal{N}_{i,1},\mathcal{N}_{i,2},\ldots,\mathcal{N}_{i,i}\}$, $i=1,2,\ldots,m$. Suppose $\mathcal{N}_{2,1}=\{\gamma_1,\gamma_2,\ldots,\gamma_k\}$, where $\gamma_1<\gamma_2<\cdots<\gamma_k$ and $\mathcal{N}_{2,2}=\{\delta_1,\delta_2,\ldots,\delta_{m-k}\}$, where $\delta_1<\delta_2<\cdots<\delta_{m-k}$. Without loss of generality, we may assume that $k<m-1$, since if $k=m-1$, we may relabel the set $\mathcal{N}_{2,1}$ as $\mathcal{N}_{2,2}$ and vice versa. Let $\alpha_1,\alpha_2,\ldots,\alpha_{k-1}$---where $\alpha_1<\alpha_2<\cdots<\alpha_{k-1}$---be those indices in $\{3,4,\ldots,m\}$ such that $\mathcal{N}^*_{\alpha_i},\mathcal{N}^{**}_{\alpha_i}\subseteq\mathcal{N}_{2,1}$. Similarly, Let $\beta_1,\beta_2,\ldots,\beta_{m-k-1}$---where $\beta_1<\beta_2<\cdots<\beta_{m-k-1}$---be those indices in $\{3,4,\ldots,m\}$ such that $\mathcal{N}^*_{\beta_i},\mathcal{N}^{**}_{\beta_i}\subseteq\mathcal{N}_{2,2}$.
	
	Let $S_1$ be the $k\times(k+1)$ submatrix of $R$ obtained by selecting rows $\gamma_1,\gamma_2,\ldots,\gamma_k$ and columns $1,2,\alpha_1,\alpha_2,\ldots,\alpha_{k-1}$.  Let $u$ denote the first column of $S_1$. Similarly, let $S_2$ be the $(m-k)\times(m-k+1)$ submatrix of $R$ obtained by selecting rows $\delta_1,\delta_2,\ldots,\delta_{m-k}$ and columns $1,2,\beta_1,\beta_2,\ldots,\beta_{m-k-1}$. We denote the first column of $S_2$ by $v$. By (\ref{eq:SoulesColumns}), either the second column of $S_1$ is $(||v||/||u||)u$ and the second column of $S_2$ is $-(||u||/||v||)v$ or the second column of $S_1$ is $-(||v||/||u||)u$ and the second column of $S_2$ is $(||u||/||v||)v$. Without loss of generality, we may assume the former, as otherwise we may replace $r_2$ with $-r_2$. Hence we may write
	\[
		S_1=\left[
			\begin{array}{ccc}
				u & \frac{||v||}{||u||}u & T_1
			\end{array}
		\right]
		\:\:\:\text{ and }\:\:\:
		S_2=\left[
			\begin{array}{ccc}
				v & -\frac{||u||}{||v||}v & T_2
			\end{array}
		\right],
	\]
	where $T_1$ and $T_2$ are $k\times(k-1)$ and $(m-k)\times(m-k-1)$ matrices, respectively.
	
	We now define the $(k+1)\times(k+1)$ matrix
	\[
		R_1:=\begin{pmat}[{.....}]
			&& S_1 &&& \cr\-
			||v|| & -||u|| & 0 & 0 & \cdots & 0 \cr
		\end{pmat}
		=\begin{pmat}[{||...}]
			u & \frac{||v||}{||u||}u && T_1 && \cr\-
			||v|| & -||u|| & 0 & 0 & \cdots & 0 \cr
		\end{pmat}
	\]
	and the $(m-k)\times(m-k)$ matrix
	\[
		R_2:=\left[
			\begin{array}{cc}
				\frac{v}{||v||} & T_2
			\end{array}
		\right].
	\]
	Note that $||u||^2+||v||^2=||x||^2=1$ and so $R_1$ and $R_2$ have normalised, positive first columns. Moreover, $R_1$ and $R_2$ are Soules matrices. To see this, consider the sequence $\tilde{\mathcal{N}}=(\tilde{\mathcal{N}}_1,\tilde{\mathcal{N}}_2,\ldots,\tilde{\mathcal{N}}_{k+1})$ of partitions of $\{1,2,\ldots,k+1\}$, where $\tilde{\mathcal{N}}_1=\{\{1,2,\ldots,k+1\}\}$, $\tilde{\mathcal{N}}_2=\{\{1,2,\ldots,k\},\{k+1\}\}$ and for $i\in\{3,4,\ldots,k+1\}$, $\tilde{\mathcal{N}}_{i,j}:=\{s:\gamma_s\in\tilde{\mathcal{M}}_{i,j}\}$, where $\tilde{\mathcal{M}}_i$ is obtained from $\mathcal{N}_{\alpha_{i-2}}$ by removing those sets $\mathcal{N}_{\alpha_{i-2},l}\in\mathcal{N}_{\alpha_{i-2}}$ which are subsets of $\mathcal{N}_{2,2}$. Then, labelling the columns of $R_1$ as
	$
		y=\tilde{r}_1,\tilde{r}_2,\ldots,${ }$\tilde{r}_{k+1},
	$
	we see that
	\[
	 \tilde{r}_i=\pm\frac{1}{\sqrt{||y_\mathcal{N}^{(i)}||^2+||\hat{y}_\mathcal{N}^{(i)}||^2}}\left( \frac{||\hat{y}_\mathcal{N}^{(i)}||}{||y_\mathcal{N}^{(i)}||}y_\mathcal{N}^{(i)}-\frac{||y_\mathcal{N}^{(i)}||}{||\hat{y}_\mathcal{N}^{(i)}||}\hat{y}_\mathcal{N}^{(i)} \right),
	\]
	for all $i=2,3,\ldots,k+1$ and hence $R_1$ is a Soules matrix. Similarly, $R_2$ may be seen to be a Soules matrix by considering the sequence $\mathcal{N}'=(\mathcal{N}_1',\mathcal{N}_2',\ldots,\mathcal{N}_{m-k}')$ of partitions of $\{1,2,\ldots,m-k\}$, where $\mathcal{N}_1'=\{\{1,2,\ldots,${ }$m-k\}\}$  and for $i\in\{2,3,\ldots,m-k\}$, $\mathcal{N}_{i,j}':=\{s:\delta_s\in\mathcal{M}'_{i,j}\}$, where $\mathcal{M}'_i$ is obtained from $\mathcal{N}_{\beta_{i-1}}$ by removing those sets $\mathcal{N}_{\beta_{i-1},l}\in\mathcal{N}_{\beta_{i-1}}$ which are subsets of $\mathcal{N}_{2,1}$.
	
	Now set
	\begin{equation}\label{eq:cDefinition}
		c:=||v||^2\lambda_1+||u||^2\lambda_2
	\end{equation}
	and observe that if
	\[
	\Lambda_1:=\text{diag}(\lambda_1,\lambda_2,\lambda_{\alpha_1},\lambda_{\alpha_2},\ldots,\lambda_{\alpha_{k-1}}),
	\]
	then the diagonal elements of the matrix $R_1\Lambda_1R_1^T$ are $(a_{\gamma_1},a_{\gamma_2},\ldots,${ }$a_{\gamma_k},c)$. To see this, note that for $i=1,2,\ldots,k$, the definitions of $\gamma_i$ and $\alpha_i$ and the structure of $R$ imply
	\[
		\sum_{t=3}^m r_{\gamma_i t}^2 \lambda_t=\sum_{s=3}^{k+1}r_{\gamma_i\alpha_{s-2}}^2\lambda_{\alpha_{s-2}}
	\]
	and hence
	\begin{align*}
		(R_1\Lambda_1 R_1^T)_{ii} &=\sum_{s=1}^{k+1}(R_1)_{is}^2(\Lambda_1)_{ss} \\
		&=\sum_{s=1}^{k+1}(S_1)_{is}^2(\Lambda_1)_{ss} \\
		&=r_{\gamma_i1}^2\lambda_1+r_{\gamma_i2}^2\lambda_2+\sum_{s=3}^{k+1}r_{\gamma_i\alpha_{s-2}}^2\lambda_{\alpha_{s-2}} \\
		&=r_{\gamma_i1}^2\lambda_1+r_{\gamma_i2}^2\lambda_2+\sum_{t=3}^m r_{\gamma_i t}^2 \lambda_t \\
		&=(R\Lambda R^T)_{\gamma_i\gamma_i} \\
		&=a_{\gamma_i}.
	\end{align*}
	In addition,
	\begin{align*}
		(R_1\Lambda_1 R_1^T)_{k+1,k+1} &=\sum_{s=1}^{k+1}(R_1)_{k+1,s}^2(\Lambda_1)_{ss} \\
		&=||v||^2\lambda_1+||u||^2\lambda_2 \\
		&=c.
	\end{align*}
	Therefore, by the inductive hypothesis,
	\begin{equation}\label{eq:SoulesEquivalence1}
		(\lambda_1,\lambda_2,\lambda_{\alpha_1},\lambda_{\alpha_2},\ldots,\lambda_{\alpha_{k-1}})\in\mathcal{H}_{k+1}(a_{\gamma_1},a_{\gamma_2},\ldots,a_{\gamma_k},c).
	\end{equation}
	Similarly, if $\Lambda_2:=\text{diag}(c,\lambda_{\beta_1},\lambda_{\beta_2},\ldots,\lambda_{\beta_{m-k-1}})$, then the diagonal elements of the matrix $R_2\Lambda_2R_2^T$ are $(a_{\delta_1},a_{\delta_2},\ldots,a_{\delta_{m-k}})$. To see this, note that for $i=1,2,\ldots,m-k$, the definitions of $\delta_i$ and $\beta_i$ and the structure of $R$ imply
	\[
		\sum_{t=3}^mr_{\delta_it}^2\lambda_t=\sum_{s=3}^{m-k+1}r_{\delta_i\beta_{s-2}}^2\lambda_{\beta_{s-2}}
	\]
	and hence
	\begin{align*}
		(R_2\Lambda_2R_2^T)_{ii} &=\sum_{s=1}^{m-k}(R_2)_{is}^2(\Lambda_2)_{ss} \\
		&=\frac{v_i^2}{||v||^2}c+\sum_{s=2}^{m-k}(T_2)_{i,s-1}^2\lambda_{\beta_{s-1}} \\
		&=\frac{v_i^2}{||v||^2}c+\sum_{s=3}^{m-k+1}(T_2)_{i,s-2}^2\lambda_{\beta_{s-2}} \\
		&=v_i^2\lambda_1+\frac{||u||^2}{||v||^2}v_i^2\lambda_2+\sum_{s=3}^{m-k+1}(S_2)_{i,s}^2\lambda_{\beta_{s-2}} \\
		&=r_{\delta_i,1}^2\lambda_1+r_{\delta_i,2}^2\lambda_2+\sum_{s=3}^{m-k+1}r_{\delta_i,\beta_{s-2}}^2\lambda_{\beta_{s-2}} \\
		&=r_{\delta_i,1}^2\lambda_1+r_{\delta_i,2}^2\lambda_2+\sum_{t=3}^mr_{\delta_it}^2\lambda_t \\
		&=(R\Lambda R^T)_{\delta_i\delta_i} \\
		&=a_{\delta_i}.
	\end{align*}
	Note also that $c\geq\lambda_2\geq\lambda_{\beta_1}$ and hence the inductive hypothesis gives
	\begin{equation}\label{eq:SoulesEquivalence2}
		(c;\lambda_{\beta_1},\lambda_{\beta_2},\ldots,\lambda_{\beta_{m-k-1}})\in\mathcal{H}_{m-k}(a_{\delta_1},a_{\delta_2},\ldots,a_{\delta_{m-k}}).
	\end{equation}
	Therefore, by Definition \ref{def:HDefinition}, (\ref{eq:SoulesEquivalence1}) and (\ref{eq:SoulesEquivalence2}) imply
	\begin{equation}\label{eq:Direction1Proved}
		(\lambda_1,\lambda_2,\ldots,\lambda_m)\in\mathcal{H}_m(a_1,a_2,\ldots,a_m).
	\end{equation}
	We have now shown that $\mathcal{S}_n(a_1,a_2,\ldots,a_n)\subseteq\mathcal{H}_n(a_1,a_2,\ldots,a_n)$.
		
	\underline{\textsc{Part 2}:} We now claim $\mathcal{H}^*_n(a_1,a_2,\ldots,a_n)\subseteq\mathcal{S}_n(a_1,a_2,\ldots,a_n)$. If $n=1$, there is nothing to prove. If $(\lambda_1;\lambda_2)\in\mathcal{H}_2(a_1,a_2)$, then there exists $\epsilon\geq0$ such that $\lambda_1=a_1+\epsilon$ and $\lambda_2=a_2-\epsilon$. If $(\lambda_1;\lambda_2)\in\mathcal{H}^*_2(a_1,a_2)$, then $\epsilon>0$ and hence the matrix
	\[
		R_0:=\frac{1}{\sqrt{a_1-a_2+2\epsilon }}\left[
			\begin{array}{cc}
				\sqrt{a_1-a_2+\epsilon } & \sqrt{\epsilon } \\
				\sqrt{\epsilon } & -\sqrt{a_1-a_2+\epsilon }
			\end{array}
		\right]
	\]
	is a Soules matrix with
	\[
		R_0\Lambda R_0^T=\left[
			\begin{array}{cc}
				a_1 & \sqrt{\epsilon \left(a_1-a_2+\epsilon \right)} \\
				\sqrt{\epsilon \left(a_1-a_2+\epsilon \right)} & a_2
			\end{array}
		\right],
	\]
	where $\Lambda:=\mathrm{diag}(\lambda_1,\lambda_2)$. Therefore $(\lambda_1;\lambda_2)\in\mathcal{S}_2(a_1,a_2)$. Now assume the claim holds for $n=m-1$, $m\geq3$, and consider the case when $n=m$.
	
	Suppose $(\lambda_1;\lambda_2,\ldots,\lambda_m)\in\mathcal{H}^*_m(a_1,a_2,\ldots,a_m)$, where $\lambda_1\geq\lambda_2\geq\cdots\geq\lambda_m$. Then by Theorem \ref{thm:HSimplification}, there exist $s,t\in\{1,2,\ldots,m\}$, $s<t$, such that
	\begin{multline}\label{eq:NonStar1}
		(\lambda_1;\lambda_2,\ldots,\lambda_{m-1})\in \\ \mathcal{H}_{m-1}(a_1,\ldots,a_{s-1},a_{s+1},\ldots,a_{t-1},a_{t+1},\ldots,a_m,c)
	\end{multline}
	and
	\begin{equation}\label{eq:NonStar2}
		(c;\lambda_m)\in\mathcal{H}_2(a_s,a_t),
	\end{equation}
	where $c:=a_s+a_t-\lambda_m$. It is not difficult to see that in (\ref{eq:NonStar1}) and (\ref{eq:NonStar2}), it is possible to replace $\mathcal{H}$ by $\mathcal{H}^*$. Let us assume the contrary. If 
	\begin{multline*}
		(\lambda_1;\lambda_2,\ldots,\lambda_{m-1})\not\in \\ \mathcal{H}^*_{m-1}(a_1,\ldots,a_{s-1},a_{s+1},\ldots,a_{t-1},a_{t+1},\ldots,a_m,c),
	\end{multline*}
	then there exist partitions
	\begin{gather*}
		\{1,2,\ldots,m-1\}=\{p_1,p_2,\ldots,p_l\}\cup\{q_1,q_2,\ldots,q_{m-l-1}\}, \\
		\{1,2,\ldots,m\}\setminus\{s,t\}=\{r_1,r_2,\ldots,r_{l-1}\}\cup\{s_1,s_2,\ldots,s_{m-l-1}\},
	\end{gather*}
	where $\{r_1,r_2,\ldots,r_{l-1}\}$ may be empty, such that
	\begin{equation}\label{eq:StarProve1}
		(\lambda_{p_1};\lambda_{p_2},\ldots,\lambda_{p_l})\in\mathcal{H}_l(a_{r_1},a_{r_2},\ldots,a_{r_{l-1}},c)
	\end{equation}
	and
	\begin{equation}\label{eq:StarProve2}
		(\lambda_{q_1};\lambda_{q_2},\ldots,\lambda_{q_{m-l-1}})\in\mathcal{H}_{m-l-1}(a_{s_1},a_{s_2},\ldots,a_{s_{m-l-1}}).
	\end{equation}
	In this case, (\ref{eq:StarProve1}) and (\ref{eq:NonStar2}) imply
	\begin{equation}\label{eq:StarProve3}
		(\lambda_{p_1};\lambda_{p_2},\ldots,\lambda_{p_l},\lambda_m)\in\mathcal{H}_{l+1}(a_{r_1},a_{r_2},\ldots,a_{r_{l-1}},a_s,a_t);
	\end{equation}
	however, (\ref{eq:StarProve2}) and (\ref{eq:StarProve3}) contradict the fact that $(\lambda_1;\lambda_2,\ldots,\lambda_m)\in\mathcal{H}^*_m(a_1,a_2,${ }$\ldots,a_m)$.
	Hence
	\begin{multline}\label{eq:StarProve4}
		(\lambda_1;\lambda_2,\ldots,\lambda_{m-1})\in \\ \mathcal{H}^*_{m-1}(a_1,\ldots,a_{s-1},a_{s+1},\ldots,a_{t-1},a_{t+1},\ldots,a_m,c).
	\end{multline}
	Similarly, suppose $(c;\lambda_m)\in\mathcal{H}_2(a_s,a_t)\setminus\mathcal{H}^*_2(a_s,a_t)$ and assume, without loss of generality, that $a_s\geq a_t$. Then $c=a_s$ and $\lambda_m=a_t$, and thus, replacing $c$ by $a_s$ in (\ref{eq:NonStar1}), we have
	\begin{gather*}
		(\lambda_1;\lambda_2,\ldots,\lambda_{m-1})\in\mathcal{H}_{m-1}(a_1,\ldots,a_{t-1},a_{t+1},\ldots,a_m), \\
		(\lambda_m)\in\mathcal{H}_1(a_t),
	\end{gather*}
	again contradicting the fact that $(\lambda_1;\lambda_2,\ldots,\lambda_m)\in\mathcal{H}^*_m(a_1,a_2,\ldots,${ }$a_m)$. Thus
	\begin{equation}\label{eq:StarProve5}
		(c;\lambda_m)\in\mathcal{H}^*_2(a_s,a_t).
	\end{equation}
	
	Applying the inductive hypothesis to (\ref{eq:StarProve4}), there exists an $(m-1)\times(m-1)$ Soules matrix $R_1$, such that the matrix $R_1\Lambda_1R_1^T$---where $\Lambda_1:=\mathrm{diag}(\lambda_1,\lambda_2,\ldots,\lambda_{m-1})$---has diagonal elements
	\[
		(a_1,\ldots,a_{s-1},a_{s+1},\ldots,a_{t-1},a_{t+1},\ldots,a_m,c).
	\]
	Note that if $S$ is a $k\times k$ Soules matrix and $P$ is a $k\times k$ permutation matrix, then $(PR)\Theta(PR)^T=P(R\Theta R^T)P^T$ is nonnegative for every nonnegative diagonal matrix $\Theta$ with non-increasing diagonal entries. Hence $PR$ is a Soules matrix also. Therefore, we may assume, without loss of generality, that $c$ is the $(m-1,m-1)$ entry of $R_1\Lambda_1R_1^T$, since otherwise, we may replace $R_1$ with $PR_1$, where $P$ is a suitable permutation matrix. Similarly, by (\ref{eq:StarProve5}), there exists a $2\times2$ Soules matrix $R_2$ such that the matrix $R_2\Lambda_2 R_2^T$---where $\Lambda_2:=\mathrm{diag}(c,\lambda_m)$---has diagonal elements $(a_s,a_t)$. Let $R_1$ be partitioned as
	\[
		R_1=\left[
			\begin{array}{c}
				U \\ u^T
			\end{array}
		\right],
	\]
	where $u\in\mathbb{R}^{m-1}$ and $U\in\mathbb{R}^{(m-2)\times(m-1)}$ and let $R_2$ be partitioned as
	\[
		R_2=\left[
			\begin{array}{cc}
				v & v'
			\end{array}
		\right]
	\]
	where $v,v'\in\mathbb{R}^2$.
	
	By Lemma \ref{lem:SmigocSDLemma}, for matrices
	\[
		R:=\left[
			\begin{array}{cc}
				U & 0 \\
				vu^T & v'
			\end{array}
		\right]
	\]
	and $\Lambda:=\mathrm{diag}(\lambda_1,\lambda_2,\ldots,\lambda_m)$, the matrix $R\Lambda R^T$ has diagonal elements $(a_1,a_2,\ldots,a_m)$. We will show that $R$ is a Soules matrix. To see this, let $\tilde{\mathcal{N}}=(\tilde{\mathcal{N}}_1,\tilde{\mathcal{N}}_2,\ldots,\tilde{\mathcal{N}}_{m-1})$ be the Soules-type sequence of partitions of $\{1,2,\ldots,m-1\}$ associated with $R_1$ , where $\tilde{\mathcal{N}}_i=\{\tilde{\mathcal{N}}_{i,1},\tilde{\mathcal{N}}_{i,2},\ldots,\tilde{\mathcal{N}}_{i,i}\}$. $R$ may then be seen to be a Soules matrix by considering the Soules-type sequence $\mathcal{N}=(\mathcal{N}_1,\mathcal{N}_2,\ldots,\mathcal{N}_m)$, $\mathcal{N}_i=\{\mathcal{N}_{i,1},\mathcal{N}_{i,2},\ldots,\mathcal{N}_{i,i}\}$, of partitions of $\{1,2,\ldots,m\}$, where for $i\in\{1,2,${ }$\ldots,m-1\}$, $\mathcal{N}_{i,j}$ is given by
	\[
		\left\{
			\begin{array}{ll}
				\mathcal{N}_{i,j}=\tilde{\mathcal{N}}_{i,j} & \text{ if }\: m-1\notin\tilde{\mathcal{N}}_{i,j} \\
				\mathcal{N}_{i,j}=\tilde{\mathcal{N}}_{i,j}\cup\{m\} & \text{ if }\: m-1\in\tilde{\mathcal{N}}_{i,j} \\
			\end{array}
		\right.
	\]
	and $\mathcal{N}_m=\{\{1\},\{2\},\ldots,\{m\}\}$.
	
	Therefore $(\lambda_1;\lambda_2,\ldots,\lambda_m)\in\mathcal{S}_m(a_1,a_2,\ldots,a_m)$. Hence $\mathcal{H}^*_n(a_1,a_2,\ldots,${ }$a_n)\subseteq\mathcal{S}_n(a_1,a_2,\ldots,a_n)$.
	
	\underline{\textsc{Part 3}:} We will now use parts 1 and 2 to show that $\overline{\mathcal{S}}_n(a_1,a_2,\ldots,${ }$a_n)=\mathcal{H}_n(a_1,a_2,\ldots,a_n)$. By definition, if $(\lambda_1;\lambda_2,\ldots,\lambda_n)\in\overline{\mathcal{S}}_n(a_1,${ }$a_2,\ldots,a_n)$, then there exist partitions
	\begin{equation}\label{eq:PartitionForm}
		\begin{array}{c}
			\{1,\ldots,n\}=\{\alpha_1^{(1)},\ldots,\alpha_{n_1}^{(1)}\}\cup\{\alpha_1^{(2)},\ldots,\alpha_{n_2}^{(2)}\}\cup\cdots\cup\{\alpha_1^{(k)},\ldots,\alpha_{n_k}^{(k)}\}, \\
			\{1,\ldots,n\}=\{\beta_1^{(1)},\ldots,\beta_{n_1}^{(1)}\}\cup\{\beta_1^{(2)},\ldots,\beta_{n_2}^{(2)}\}\cup\cdots\cup\{\beta_1^{(k)},\ldots,\beta_{n_k}^{(k)}\},
		\end{array}
	\end{equation}
	such that
	\begin{equation}\label{eq:SSplit}
		\left(\lambda_{\alpha_1^{(i)}};\lambda_{\alpha_2^{(i)}},\ldots,\lambda_{\alpha_{n_i}^{(i)}}\right)\in\mathcal{S}_{n_i}\left(a_{\beta_1^{(i)}},a_{\beta_2^{(i)}},\ldots,a_{\beta_{n_i}^{(i)}}\right) :\hspace{3mm} i=1,2,\ldots,k.
	\end{equation}
	By Part 1, in (\ref{eq:SSplit}), we may replace $\mathcal{S}$ by $\mathcal{H}$ and hence, by Observation \ref{obv:UnionForSoules}, $(\lambda_1;\lambda_2,\ldots,\lambda_n)\in\mathcal{H}_n(a_1,a_2,\ldots,a_n)$.
	
	Conversely, if $(\lambda_1;\lambda_2,\ldots,\lambda_n)\in\mathcal{H}_n(a_1,a_2,\ldots,a_n)$, then there exist partitions of the form (\ref{eq:PartitionForm}) such that
	\begin{equation}\label{eq:HSplit}
		\left(\lambda_{\alpha_1^{(i)}};\lambda_{\alpha_2^{(i)}},\ldots,\lambda_{\alpha_{n_i}^{(i)}}\right)\in\mathcal{H}^*_{n_i}\left(a_{\beta_1^{(i)}},a_{\beta_2^{(i)}},\ldots,a_{\beta_{n_i}^{(i)}}\right) : i=1,2,\ldots,k.
	\end{equation}
	By Part 2, in (\ref{eq:HSplit}), we may replace $\mathcal{H}^*$ by $\mathcal{S}$ and hence, by definition, $(\lambda_1;\lambda_2,\ldots,\lambda_n)\in\overline{\mathcal{S}}_n(a_1,a_2,\ldots,a_n)$.
	
	\underline{\textsc{Part 4}:} We now claim that (ii) implies (iv). If $(\lambda)\in\mathcal{H}_1$, then $\lambda\geq0$ and $(\lambda)$ trivially satisfies $\mathbb{S}_1$. Now assume the claim holds for $n\in\{1,2,\ldots,m-1\}$ and suppose $\sigma:=(\lambda_1,\lambda_2,\ldots,\lambda_m)\in\mathcal{H}_m$, where $\lambda_1\geq\lambda_2\geq\cdots\geq\lambda_m$. Then by Theorem \ref{thm:AntiGuoForSoules}, there exist
	$
		\epsilon\in\left[0,\frac{1}{2}(\lambda_1-\lambda_2)\right]
	$
	and a partition
	\[
		\{3,4,\ldots,m\}=\{p_1,p_2,\ldots,p_{l-1}\}\cup\{q_1,q_2,\ldots,q_{m-l-1}\},
	\]
	where $\{p_1,p_2,\ldots,p_{l-1}\}$ or $\{q_1,q_2,\ldots,q_{m-l-1}\}$ may be empty, such that
	\[
		(\lambda_1-\epsilon,\lambda_{p_1},\lambda_{p_2},\ldots,\lambda_{p_{l-1}})\in\mathcal{H}_l
	\]
	and
	\[
		(\lambda_2+\epsilon,\lambda_{q_1},\lambda_{q_2},\ldots,\lambda_{q_{m-l-1}})\in\mathcal{H}_{m-l}.
	\]
	By the inductive hypothesis, there exist $k$ and $t$ such that
	\[
		\sigma^*_1:=(\lambda_1-\epsilon,\lambda_{p_1},\lambda_{p_2},\ldots,\lambda_{p_{l-1}})
	\]
	and
	\[	
		\sigma^*_2:=(\lambda_2+\epsilon,\lambda_{q_1},\lambda_{q_2},\ldots,\lambda_{q_{m-l-1}})
	\]
	satisfy $\mathbb{S}_k$ and $\mathbb{S}_t$, respectively. Furthermore, by Observation \ref{obv:p_implies_p+1}, we may assume $k=t$, i.e. (in the notation of Section \ref{sec:Sp}) $\sigma_1^*,\sigma_2^*\in\mathcal{Q}_{\mathbb{S}_k}$.
	
	Define also
	\begin{gather*}
		\sigma_1:=(\lambda_1,\lambda_{p_1},\lambda_{p_2},\ldots,\lambda_{p_{l-1}}), \\
		\sigma_2:=(\lambda_2,\lambda_{q_1},\lambda_{q_2},\ldots,\lambda_{q_{m-l-1}}).
	\end{gather*}
	Recalling definitions (\ref{eq:SotoMDef}) and (\ref{eq:SotoNDef}), we have
	$
		\mathcal{M}_{\mathbb{S}_k}(\sigma_1)=\mathcal{M}_{\mathbb{S}_k}(\sigma_1^*)+\epsilon\geq\epsilon
	$
	,
	$
		\mathcal{N}_{\mathbb{S}_k}(\sigma_2)\leq\epsilon
	$
	and
	\[
		\gamma:=\max\{\lambda_1-\mathcal{M}_{\mathbb{S}_k}(\sigma_1),\lambda_2\}\leq\max\{\lambda_1-\epsilon,\lambda_2\}=\lambda_1-\epsilon.
	\]
	Since $\lambda_1\geq\gamma+\mathcal{N}_{\mathbb{S}_k}(\sigma_2)$, $\sigma=(\sigma_1,\sigma_2)$ satisfies $\mathbb{S}_{k+1}$.
	
	\underline{\textsc{Part 5}:} We now claim (iv) implies (iii). The cases $p=1$ and $p=2$ have been dealt with in \cite{UnifiedView}. Now assume that if $\sigma$ satisfies $\mathbb{S}_{p-1}$, then $\sigma$ is C-realisable. The proof of the inductive step is essentially the same as the proof of \cite[Theorem 3.7]{UnifiedView}.
	
	\underline{\textsc{Part 6}:} Finally, we show that (iii) implies (ii). If $(\lambda_1,\lambda_2,\ldots,\lambda_n)$ is C-realisable, then by definition, it may be obtained starting from the $n$ lists $(0),(0),\ldots,(0)$ and using only the operations defined by Observation \ref{obv:Union} and Theorems \ref{thm:PerronIncrease} and \ref{thm:Guo}. By Observation \ref{obv:UnionForSoules}, Lemma \ref{lem:SymmetricPerronLemma}, Theorem \ref{thm:GuoForSoules} and the fact that $(0)\in\mathcal{H}_1$, we see that $(\lambda_1,\lambda_2,\ldots,\lambda_n)\in\mathcal{H}_n$.
\end{proof}

\begin{cor}
	If $\sigma$ is C-realisable, then $\sigma$ is symmetrically realisable.
\end{cor}

We will illustrate Part 1 of the proof of Theorem \ref{thm:SoulesEquivalence} by considering a specific example:

\begin{ex}
	Consider again the list $\sigma=(\lambda_1,\lambda_2,\lambda_3,\lambda_4,\lambda_5)=(7,5,-2,${ }$-4,-6)$. We showed in Example \ref{ex:SoulesExample} that $\sigma\in\mathcal{S}_5$, and in Example \ref{ex:HExample} that $\sigma\in\mathcal{H}_5$.	
	
	Given the Soules matrix $R$ of Example \ref{ex:SoulesExample}, let us follow the proof of Theorem \ref{thm:SoulesEquivalence} to obtain the decomposition given in Figure \ref{fig:SigmaDecomp}. In the notation of the proof, we have $k=2$, $\mathcal{N}_{2,1}=\{\gamma_1,\gamma_2\}=\{1,2\}$, $\mathcal{N}_{2,2}=\{\delta_1,\delta_2,\delta_3\}=\{3,4,5\}$, $\alpha_1=5$, $(\beta_1,\beta_2)=(3,4)$,
	\[
		S_1=
\left[
\begin{array}{ccc}
 \frac{1}{2} & \frac{1}{2} & \frac{1}{\sqrt{2}} \\
 \frac{1}{2} & \frac{1}{2} & -\frac{1}{\sqrt{2}}
\end{array}
\right],\hspace{3mm}
		S_2=
\left[
\begin{array}{cccc}
 \frac{1}{2 \sqrt{2}} & -\frac{1}{2 \sqrt{2}} & \frac{\sqrt{3}}{2} & 0 \\
 \frac{\sqrt{3}}{4} & -\frac{\sqrt{3}}{4} & -\frac{1}{2 \sqrt{2}} & \frac{1}{\sqrt{2}} \\
 \frac{\sqrt{3}}{4} & -\frac{\sqrt{3}}{4} & -\frac{1}{2 \sqrt{2}} & -\frac{1}{\sqrt{2}}
\end{array}
\right],
	\]
	\[
	||u||=\left|\left|\left[ \begin{array}{cc} \frac{1}{2} & \frac{1}{2} \end{array} \right]^T\right|\right|=\frac{1}{\sqrt{2}},\hspace{3mm}
	||v||=\left|\left|\left[ \begin{array}{ccc} \frac{1}{2 \sqrt{2}} & \frac{\sqrt{3}}{4} & \frac{\sqrt{3}}{4} \end{array} \right]^T\right|\right|=\frac{1}{\sqrt{2}},
	\]
	\[
		R_1=\begin{pmat}[{..}]
\frac{1}{2} & \frac{1}{2} & \frac{1}{\sqrt{2}} \cr
\frac{1}{2} & \frac{1}{2} & -\frac{1}{\sqrt{2}} \cr\-
\frac{1}{\sqrt{2}} & -\frac{1}{\sqrt{2}} & 0 \cr
		\end{pmat},\hspace{3mm}
		R_2=\begin{pmat}[{|.}]
\frac{1}{2} & \frac{\sqrt{3}}{2} & 0 \cr
\sqrt{\frac{3}{8}} & -\frac{1}{2 \sqrt{2}} & \frac{1}{\sqrt{2}} \cr
\sqrt{\frac{3}{8}} & -\frac{1}{2 \sqrt{2}} & -\frac{1}{\sqrt{2}} \cr
		\end{pmat}
	\]
	and $c=||v||^2\lambda_1+||u||^2\lambda_2=6$. Thus, in order to show that $(7;5,-2,-4,${ }$-6)\in\mathcal{H}_5(0,0,0,0,0)$, it is sufficient to show that
	\begin{equation}\label{eq:SmallerSpec1}
	 	(7;5,-6)\in\mathcal{H}_3(0,0,\mathbf{6})
	\end{equation}
	and
	\begin{equation}\label{eq:SmallerSpec2}
		(\mathbf{6};-2,-4)\in\mathcal{H}_3(0,0,0).
	\end{equation}
	Since (\ref{eq:SmallerSpec1}) and (\ref{eq:SmallerSpec2}) obey the conditions given in (\ref{eq:n=3NS}), Lemma \ref{lem:n=3} implies that (\ref{eq:SmallerSpec1}) and (\ref{eq:SmallerSpec2}) hold.
\end{ex}

\section{Comparison to the literature}\label{sec:LiteratureComparison}

Over the years, many realisability criteria have been given in the literature which guarantee that a list of real numbers $\sigma$ be the spectrum of a nonnegative matrix. Consider, for example, the list of sufficient conditions given in Table \ref{tab:SufficientConditions}. In this section, we demonstrate that if $\sigma:=(\lambda_1;\lambda_2,\ldots,\lambda_n)$ satisfies any of these criteria, then $\sigma\in\mathcal{H}_n$.

\begin{table}[hbt]
	\centering
	\begin{tabular*}{\textwidth}{@{\extracolsep{\fill}} crcrcc }
		\toprule
		\hspace{2mm} & & Author & Year & See & \hspace{2mm} \\ 
		\midrule
		& 1. & Suleimanova  & 1949 & \cite{Suleimanova1949} & \\
		& 2. & Perfect & 1953 & \cite{Perfect1953} & \\
		& 3. & Ciarlet & 1968 & \cite{Ciarlet1968} & \\
		& 4. & Kellog & 1971 & \cite{Kellog1971} & \\
		& 5. & Salzmann & 1972 & \cite{Salzmann1972} & \\
		& 6. & Fiedler & 1974 & \cite[Theorem 2.1]{Fiedler} & \\
		& 7. & Borobia & 1995 & \cite{Borobia1995} & \\
		& 8. & Soto & 2003 & \cite[Theorem 11]{Soto2003} / $\mathbb{S}_1$ & \\
		& 9. & Soto & 2003 & \cite[Theorem 17]{Soto2003} / $\mathbb{S}_2$ & \\
		\bottomrule
	\end{tabular*}
	\caption{Sufficient conditions for the RNIEP}
	\label{tab:SufficientConditions}
\end{table}

In \cite{Map}, the authors prove that if $\sigma$ satisfies any of the conditions 1--9, then $\sigma$ must satisfy either Condition 7 or Condition 9. Radwan \cite{Radwan} showed that Condition 7 is sufficient for the existence of a symmetric nonnegative matrix with spectrum $\sigma$ and Soto \cite{Soto2006} showed that the same is true of Condition 9. Hence all of the conditions 1--9 are sufficient for the SNIEP. Moreover, in \cite{UnifiedView}, the authors show that if $\sigma$ obeys either Condition 7 or Condition 9, then $\sigma$ is C-realisable. Therefore, by Theorem \ref{thm:SoulesEquivalence}, if $\sigma:=(\lambda_1,\lambda_2,\ldots,\lambda_n)$ satisfies any of the criteria 1--9, then $\sigma\in\mathcal{H}_n$.

Finally, note that $\mathcal{H}_n$ is more general than any of the Conditions 1--9. The list $\sigma:=(25,21,18,16,-10,-10,-10,-10,-10,-10,-10,-10)$ was shown to be C-realisable in \cite{UnifiedView} and was shown to satisfy $\mathbb{S}_3$ in \cite{Soto2013} (hence $\sigma\in\mathcal{H}_{12}$ by Theorem \ref{thm:SoulesEquivalence}), but $\sigma$ does not satisfy any of the conditions 1--9.

\appendix

\bibliographystyle{elsarticle-num}
\bibliography{Bibliography}

\begin{thebibliography}{10}
\expandafter\ifx\csname url\endcsname\relax
  \def\url#1{\texttt{#1}}\fi
\expandafter\ifx\csname urlprefix\endcsname\relax\def\urlprefix{URL }\fi
\expandafter\ifx\csname href\endcsname\relax
  \def\href#1#2{#2} \def\path#1{#1}\fi

\bibitem{MR3217406}
G.~M. Engel, H.~Schneider, S.~Sergeev, On sets of eigenvalues of matrices with
  prescribed row sums and prescribed graph, Linear Algebra Appl. 455 (2014)
  187--209.

\bibitem{MR1780191}
M.~Arav, D.~Hershkowitz, V.~Mehrmann, H.~Schneider, The recursive inverse
  eigenvalue problem, SIAM J. Matrix Anal. Appl. 22~(2) (2000) 392--412
  (electronic).

\bibitem{SmigocDiagonalElement}
H.~\v{S}migoc, The inverse eigenvalue problem for nonnegative matrices, Linear
  Algebra and its Applications 393 (2004) 365 -- 374, (Special issue on
  Positivity in Linear Algebra).

\bibitem{Soules}
G.~Soules, Constructing symmetric nonnegative matrices, Linear and Multilinear
  Algebra 13 (1983) 241--251.

\bibitem{ElsnerNabbenNeumann}
L.~Elsner, R.~Nabben, M.~Neumann, Orthogonal bases that lead to symmetric
  nonnegative matrices, Linear Algebra and its Applications 271~(1--3) (1998)
  323--343.

\bibitem{UnifiedView}
A.~Borobia, J.~Moro, R.~L. Soto, A unified view on compensation criteria in the
  real nonnegative inverse eigenvalue problem, Linear Algebra and its
  Applications 428~(11--12) (2008) 2574--2584.

\bibitem{Soto2013}
R.~L. Soto, A family of realizability criteria for the real and symmetric
  nonnegative inverse eigenvalue problem, Numerical Linear Algebra with
  Applications 20~(2) (2013) 336--348.

\bibitem{Suleimanova1949}
H.~Sule{\v i}manova, Stochastic matrices with real characteristic values, Dokl.
  Akad. Nauk. S.S.S.R. 66 (1949) 343--345, (in Russian).

\bibitem{Perfect1953}
H.~Perfect, Methods of constructing certain stochastic matrices, Duke Math. J.
  20 (1953) 395--404.

\bibitem{Ciarlet1968}
P.~Ciarlet, Some results in the theory of nonnegative matrices, Linear Algebra
  and its Applications 1~(1) (1968) 139--152.

\bibitem{Kellog1971}
R.~Kellogg, Matrices similar to a positive or essentially positive matrix,
  Linear Algebra and its Applications 4 (1971) 191--204.

\bibitem{Salzmann1972}
F.~Salzmann, A note on the eigenvalues of nonnegative matrices, Linear Algebra
  and its Applications 5 (1972) 329--338.

\bibitem{Fiedler}
M.~Fiedler, Eigenvalues of nonnegative symmetric matrices, Linear Algebra and
  its Applications 9 (1974) 119--142.

\bibitem{Borobia1995}
A.~Borobia, On the nonnegative eigenvalue problem, Linear Algebra and its
  Applications 223--224 (1995) 131--140, (Special issue honouring Miroslav
  Fiedler and Vlastimil Ptak).

\bibitem{Soto2003}
R.~L. Soto, Existence and construction of nonnegative matrices with prescribed
  spectrum, Linear Algebra and its Applications 369 (2003) 169--184.

\bibitem{McDonaldNeumann}
J.~McDonald, M.~Neumann, The {S}oules approach to the inverse eigenvalue
  problem for nonnegative symmetric matrices of order $n\leq5$, Contemporary
  Math 259 (2000) 387--407.

\bibitem{LoewyMcDonald}
R.~Loewy, J.~McDonald, The symmetric nonnegative inverse eigenvalue problem for
  $5\times5$ matrices, Linear Algebra and its Applications 393 (2004) 275--298,
  (Special issue on Positivity in Linear Algebra).

\bibitem{CHN2006}
M.~Q. Chen, L.~Han, M.~Neumann, On single and double {S}oules matrices, Linear
  Algebra and its Applications 416~(1) (2006) 88--110, (Special Issue devoted
  to the Haifa 2005 conference on matrix theory).

\bibitem{CNS2007}
M.~Q. Chen, M.~Neumann, N.~Shaked-Monderer, Double {S}oules pairs and matching
  {S}oules bases, Linear Algebra and its Applications 421~(2--3) (2007)
  315--327, (Special Issue in honour of Miroslav Fiedler).

\bibitem{Nabben2007}
R.~Nabben, On relationships between several classes of z-matrices, m-matrices
  and nonnegative matrices, Linear Algebra and its Applications 421~(2--3)
  (2007) 417--439, (Special Issue in honour of Miroslav Fiedler).

\bibitem{CNS2008}
M.~Q. Chen, M.~Neumann, N.~Shaked-Monderer, Basic {S}oules matrices and their
  applications, Linear Algebra and its Applications 429~(5--6) (2008) 951--971,
  (Special Issue devoted to selected papers presented at the 13th Conference of
  the International Linear Algebra Society).

\bibitem{EubanksMcDonald2009}
S.~D. Eubanks, J.~J. McDonald, On a generalization of {S}oules bases, SIAM
  Journal on Matrix Analysis and Applications 31~(3) (2009) 1227--1234.

\bibitem{CHNP2004}
M.~Catral, L.~Han, M.~Neumann, R.~Plemmons, On reduced rank nonnegative matrix
  factorization for symmetric nonnegative matrices, Linear Algebra and its
  Applications 393 (2004) 107--126, (Special Issue on Positivity in Linear
  Algebra).

\bibitem{Naomi2004}
N.~Shaked-Monderer, A note on the cp-rank of matrices generated by {S}oules
  matrices, Electronic Journal of Linear Algebra 12 (2004) 2--5.

\bibitem{SmigocSubmatrixConstruction}
H.~\v{S}migoc, Construction of nonnegative matrices and the inverse eigenvalue
  problem, Linear and Multilinear Algebra 53~(2) (2005) 85--96.

\bibitem{NewListsFromOld}
R.~Ellard, H.~\v{S}migoc, Constructing new realisable lists from old in the
  {NIEP}, Linear Algebra and its Applications 440 (2014) 218--232.

\bibitem{Guo}
W.~Guo, Elgenvalues of nonnegative matrices, Linear Algebra and its
  Applications 266 (1997) 261 -- 270.

\bibitem{LSSymmetric}
T.~J. Laffey, H.~\v{S}migoc, Construction of nonnegative symmetric matrices
  with given spectrum, Linear Algebra and its Applications 421~(1) (2007) 97 --
  109, (Special issue devoted to the 12th Conference of the International
  Linear Algebra Society).

\bibitem{BoyleHandelman}
M.~Boyle, D.~Handelman, The spectra of nonnegative matrices via symbolic
  dynamics, Annals of Mathematics 133~(2) (1991) 249--316.

\bibitem{LaffeyConstructive}
T.~J. Laffey, A constructive version of the {B}oyle-{H}andelman theorem on the
  spectra of nonnegative matrices, Linear Algebra and its Applications 436~(6)
  (2012) 1701--1709, (Special issue in Honour of Jos\'{e} Perdig\~{a}o Dias da
  Silva).

\bibitem{RealSymmetricDifferent}
C.~R. Johnson, T.~J. Laffey, R.~Loewy, The real and the symmetric nonnegative
  inverse eigenvalue problems are different, Proc. Amer. Math. Soc. 124~(12)
  (1996) 3647--3651.

\bibitem{Map}
C.~Mariju{\'a}n, M.~Pisonero, R.~L. Soto, A map of sufficient conditions for
  the real nonnegative inverse eigenvalue problem, Linear Algebra and its
  Applications 426~(2--3) (2007) 690--705.

\bibitem{Radwan}
N.~Radwan, An inverse eigenvalue problem for symmetric and normal matrices,
  Linear Algebra and its Applications 248 (1996) 101--109.

\bibitem{Soto2006}
R.~L. Soto, Realizability criterion for the symmetric nonnegative inverse
  eigenvalue problem, Linear Algebra and its Applications 416~(2--3) (2006)
  783--794.

\end{thebibliography}

\end{document}